%% file: mouse_set_past_projective.tex
\documentclass[oneside,12pt]{amsart}

\usepackage{amsmath,amssymb,latexsym,eucal,amsthm,rotating}
\input{mathdefs}

\DeclareMathOperator{\NSat}{NSat}
\input{theoremstyles}

\begin{document}

\title{The Mouse Set Theorem Just Past Projective}
\author{Mitch Rudominer}


\begin{abstract}
We identify a particular mouse, $\Mladder$, the minimal ladder mouse,
that sits in the mouse order just past $M_n^{\sharp}$ for all $n$,
 and we show that $\R\intersect \Mladder = Q_{\omega+1}$,
 the set of reals that are
$\Delta^1_{\omega+1}$ in a countable ordinal. Thus $Q_{\omega+1}$
is a mouse set.

This is analogous to the fact that $\R\intersect M^{\sharp}_1 = Q_3$ where $M^{\sharp}_1$ is the
the sharp for the minimal inner model with a Woodin cardinal, and $Q_3$ is the set of reals
that are $\Delta^1_3$ in a countable ordinal.

More generally $\R\intersect M^{\sharp}_{2n+1} = Q_{2n+3}$.
The mouse $\Mladder$ and the set $Q_{\omega+1}$  compose the next natural
pair to consider in this series of results. Thus we are proving the mouse
set theorem just past projective.

Some of this is not new. $\R\intersect \Mladder \subseteq Q_{\omega+1}$ was
known in the 1990's. But $Q_{\omega+1} \subseteq \Mladder$ was open
until Woodin found a proof in 2018. The main goal of this paper is to give
Woodin's proof.
\end{abstract}

\maketitle

\tableofcontents

\section{Introduction}
\label{section:intro}
Throughout this paper we write $\R$ to mean $\Bairespace$ and we call elements of $\Bairespace$ \emph{reals}.

In the 1990's Martin, Steel and Woodin proved that, assuming large cardinals or determinacy,
$\R\intersect M_n = $ the set of reals that are $\Delta^1_{n+2}$ in a countable ordinal,
where $M_n$ is the standard, minimal proper-class inner model with $n$ Woodin cardinals (including the case $M_0=L$.)
For even $n$, $\R\intersect M_n = C_{n+2}$, the largest countable $\Sigma^1_{n+2}$ set.
For odd $n$, $\R\intersect M_n = Q_{n+2}$, the largest countable $\Pi^1_{n+2}$
set closed downwards under $\Delta^1_{n+2}$ degrees. See \cite{Proj_WO_In_Mod}.

Recall the definition of $n$-small (Definition 1.1 from \cite{Proj_WO_In_Mod}.) A premouse $M$ is $n$-small above $\delta$ 
iff whenever $\kappa$ is the critical point of an extender on the $M$-sequence, and $\delta<\kappa$, then
$$ \cJ^M_{\kappa}\not\models \text{there are } n \text{ Woodin cardinals } > \delta.$$
We say that M is $n$-small iff $M$ is $n$-small above $0$.

In the above-mentioned theorems, if one wishes to work with countable mice rather than proper class models, one may substitute
$M_n^{\sharp}$ for $M_n$ since the two models have the same reals. Here $M_n^{\sharp}$ is the least sound, iterable mouse
that is not $n$-small.

In \cite{My_Thesis} and \cite{Mouse_Sets} we explored an extension of this theory
to projective-like pointclasses beyond the projective and achieved partial results.
But even at the very first step past projective the full mouse set theorem was left open.

Let $\Sigma^1_{\omega}$
be the pointclass consisting of recursive unions of Projective sets. Consider
the projective hierarchy built over $\Sigma^1_{\omega}$:
$\Pi^1_{\omega} = \neg \Sigma^1_{\omega}$,
$\Sigma^1_{\omega+1} = \exists^{\R} \Pi^1_{\omega}$, etc.
By \cite{Scales_In_LofR}, $\Sigma^1_{\omega+n} = \Powerset(\R)\intersect \Sigma_{n+1}(J_2(\R))$
and we get a Second Periodicity Theorem: $\Sigma^1_{\omega+2n}$ and $\Pi^1_{\omega+2n+1}$ are scaled pointclasses.
$\Sigma^1_{\omega}$ is not closed under $\forall^{\omega}$ and so is not a perfect analog of $\Sigma^1_1$.
But $\Pi^1_{\omega+1}$ is closed under $\exists^{\omega}$ and is a good analog of $\Pi^1_1$ or $\Pi^1_3$.
Let $Q_{\omega+1}$ be the set of reals that are $\Delta^1_{\omega+1}$ in a countable ordinal.

In \cite{My_Thesis}, we defined the minimal ladder mouse $\Mladder$ and showed that
$\R\intersect\Mladder\subseteq Q_{\omega+1}$. ($\Mladder$ is the minimal iterable mouse $M$ such
that for each $n\in\omega$ there is a cardinal $\delta_n$ of $M$ and an initial segment $P_n \initseg M$
such that $\delta_n$ is Woodin in $P_n$ and $P_n$ is not $n$-small above $\delta_n$.)

But the other direction, $Q_{\omega+1} \subseteq \Mladder$, remained open until
Woodin found a proof in the summer of 2018. Woodin transmitted the proof to Steel who fleshed it out
and extended it in some hand-written notes \cite{Mouse_Pairs_and_Suslin_Cardinals}.
The author refined the relevant sections of those notes into this paper.
Woodin and Steel have elected to not appear as co-authors of this paper,
but the mathematical provenance of the ideas contained in the paper should be clear from this paragraph.

No new technology is involved in the proof and it could have been discovered in the 1990's. The main ingredients
of the proof are stationary tower forcing, genericity iterations, and a lemma stating that
all small generic extensions of a sufficiently correct, countable, transitive model of Set Theory
meet the same equivalence classes of a thin, sufficiently
definable equivalence
relation. (See Lemma \ref{hjorthlemma}.)

One of the key technical themes of the paper is that, even though
the $\delta_n$ of $\Mladder$ are only Woodin over an initial segment $P_n$
of $\Mladder$ as opposed to all of $\Mladder$, because they are
cardinals in $\Mladder$ we can recover some of the same properties
that we would have if $\angles{\delta_n}$ were an $\omega$-sequence
of Woodins in $\Mladder$. A good example of this theme is Lemma
\ref{LadderGeneric} which gives us a filter $G$ for the countable
stationary tower of $\Mladder$ such that for all $n$,
$G\intersect J^{\Mladder}_{\delta_n}$ is $\Q_{<\delta_n}^M$-generic over $P_n$, and this implies
that the ultrapower $\Ult(\Mladder, G)$ has some of the same properties as an
ultrapower by the stationary tower up to the supremum
of $\omega$ Woodin cardinals.

Beyond $\ZFC$, there are two necessary hypotheses for this paper:
\begin{enumerate}
\item that determinacy holds for the pointclasses we consider (roughly, $\bPi_2(\JofR{2})$),
\item and that the mice we consider exist and are sufficiently iterable.
\end{enumerate}
We will not attempt to articulate optimal hypotheses. Instead, for the official hypothesis for this paper we will
use: 
 $\ZFC + $ ``there exists $\omega$ Woodin cardinals with a measurable cardinal above them all.''   If we don't explicitly state otherwise, this is the assumption of all
 theorems in this paper.

 This large cardinal hypothesis implies $\AD^{\LofR}$ (see
 Theorem 8.24 of \cite{DeterminacyInLOfR}) and it implies that
 $M_{\omega}^{\sharp}$ exists and is $(\omega,\omega_1,\omega_1+1)$-iterable (see
 Theorem 7.2 of \cite{OutlineOfInnerModelTheory}.) So in particular the hypothesis gives us more
 than enough determinacy and iterable-mice existence for the purposes of this paper. 
 But the reader should keep in
 mind that the assumption of actual large cardinal in $V$ is not really
 an essential ingredient here, it is just a convenient way of arranging the
 assumptions we really need: sufficient determinacy and mice.

\section{Projective-like Pointclasses Just Past Projective}
\label{section:projectivelikepointclasses}

\begin{definition}
Let $\PiOneOmega$ be the pointclass of recursive intersections of infinitely many (lightface) projective sets.

In more detail, fix for the remainder of this paper, for $n\in\omega$, $G^n\subset\omega\times\R$
a universal $\Pi^1_{2n+1}$ set, uniformly in $n$. (``Uniformly in $n$'' here and below
means there is a recursive sequence of formula with the $n$-th formula defining the $n$-th relation.)
A $\PiOneOmega$ \emph{code} is a total recursive function $h:\omega\to\omega$. If $h$ is a $\PiOneOmega$ code then
$A^h=\setof{x\in\R}{(\forall n)\, G^n(h(n),x)}$. Say $A\subseteq\R$ is $\PiOneOmega$ iff
$A=A^h$ for some $\PiOneOmega$ code $h$. In this case we will also say that $h$
is a $\PiOneOmega$ code \emph{for} $A$.

Similarly we define $\PiOneOmega$ codes and $\PiOneOmega$ subsets of $\omega^s\times \R^t$ for $s,t\in\omega$.

Then we define $\SigmaOneOmega = \neg\PiOneOmega$, $\SigmaOneOmegaPlusOne=\exists^{\R}\PiOneOmega$,
$\PiOneOmegaPlusOne = \forall^{\R}\SigmaOneOmega$ and
$\DeltaOneOmegaPlusOne = \SigmaOneOmegaPlusOne \intersect \PiOneOmegaPlusOne$.

Finally we define as usual the relativized and bold-face
pointclasses $\PiOneOmega(x)$
and $\bPiOneOmega$ etc.
\end{definition}

The following remarks follow from \cite{Scales_In_LofR} and \cite{Mouse_Sets}.
Here we are following \cite{Scales_In_LofR} in defining $\Sigma_n(\JofR{2})$ to 
mean $\Sigma_n(\JofR{2},\singleton{V_{\omega+1}})$. Thus we are always allowed $V_{\omega+1}$
(but not its elements) as a parameter.
\begin{remarks} \
\begin{enumerate}
\item$\SigmaOneOmega = \Sigma_1(\JofR{2}) \intersect \Powerset(\R)$
\item $\PiOneOmega = \Pi_1(\JofR{2}) \intersect \Powerset(\R)$
\item $\SigmaOneOmegaPlusOne = \Sigma_2(\JofR{2}) \intersect \Powerset(\R)$
\item $\PiOneOmegaPlusOne = \Pi_2(\JofR{2}) \intersect \Powerset(\R)$
\item $\SigmaOneOmega$ and $\PiOneOmegaPlusOne$ are scaled pointclasses.
\end{enumerate}
\end{remarks}
\begin{proof}[proof sketch]
(1) Let $A\subset\R$ be $\SigmaOneOmega$. Then there is a recursive sequence of formulae $\sequence{\varphi_n}{n\in\omega}$
such that $x\in A \iff \exists n (V_{\omega+1},\epsilon)\models \varphi_n[x]$.
Since the satisfaction relation is $\Delta_1$ definable in any rudimentarily closed, transitive set,
$A\in \Sigma_1(\JofR{2})$.

That $\Sigma_1(\JofR{2}) \intersect \Powerset(\R) \subseteq \SigmaOneOmega$ follows from Lemma 1.1 of \cite{Mouse_Sets}.

(3) follows from Lemma 2.6 of \cite{Mouse_Sets} becase, using the terminology from that paper, $0$ is a good parameter
for $\JofR{2}$. (As we elaborate on in section 6, the
$\Sigma^1_{\omega+1}$ of this paper is
equal to the $\Sigma_{(2,1)}$ of \cite{Mouse_Sets}.)

(5) follows from Lemma 2.11 of \cite{Mouse_Sets}. (Again, the $\SigmaOneOmega$
of this paper is equal to the $\Sigma_{(2,0)}$ of \cite{Mouse_Sets} and
the $\PiOneOmegaPlusOne$ of this paper is equal to the $\Pi_{(2,1)}$ of \cite{Mouse_Sets}.)
\end{proof}

\begin{definition}
We say that $x\in\R$ is $\DeltaOneOmegaPlusOne$ in a countable ordinal iff there is an $\alpha<\omega_1$ such
that for all $w\in\WO$ with $|w|=\alpha$, $\singleton{x}$ is $\DeltaOneOmegaPlusOne(w)$.
Let $Q_{\omega+1} = \setof{x\in\R}{x\text{ is }\DeltaOneOmegaPlusOne\text{ in a countable ordinal}}$.
\end{definition}

By way of intuition, we consider $\PiOneOmegaPlusOne$ to be analogous to
$\Pi^1_1$ or $\Pi^1_3$ and we consider $Q_{\omega+1}$ to be analogous to
$Q_1$ or $Q_3$. These types of sets have been much studied by
descriptive set theorists and have many interesting properties.
For example $Q_3$ is the largest $\Pi^1_3$ set of reals closed downwards
under $\Delta^1_3$-degrees. See \cite{Q_Theory}.

Recall that $Q_3 = \R \intersect M_1^{\sharp}$. In the next section we will
define a mouse $\Mladder$ such that $Q_{\omega+1} = \R\intersect\Mladder$.

\section{Ladder Mice}
\label{section:laddermice}
\begin{definition}
Let $M$ be a premouse, in the sense of \cite{FSIT}. A \emph{ladder over $M$}
is a sequence of ordinals $\sequence{\delta_n,\gamma_n}{n\in\omega}$ such that
for all $n$:
\begin{enumerate}
\item $\delta_n < \gamma_n < \delta_{n+1} < \ord(M)$,
\item $\delta_n$ is a cardinal of $M$,
\item $\delta_n$ is Woodin in $J^M_{\gamma_n}$,
\item $\gamma_n$ is the least $\gamma$ such that $\cJ^M_{\gamma}$ is not $n$-small above $\delta_n$.
\end{enumerate}

To remind the reader of some notation from \cite{FSIT} and \cite{Many_Woodins},
$\cJ^M_{\gamma}$ is the model that includes a predicate for the 
last extender, if there is one indexed at $\gamma$ on the $M$-sequence,
whereas $J^M_{\gamma}$ does not include
such a predicate. In the definition above we use $\cJ^M_{\gamma}$ in (4)
because, by definition of $n$-small, the least mouse that is not $n$ small
must be active.

A ladder $\sequence{\delta_n,\gamma_n}{n\in\omega}$ over $M$ is \emph{cofinal}
iff the $\gamma_n$ are cofinal in $\ord(M)$.

$M$ is a \emph{ladder mouse} iff there is a ladder over $M$ and a
\emph{cofinal ladder mouse} iff there is a cofinal ladder over $M$.

$M$ is a \emph{minimal ladder mouse} iff $M$ is a ladder mouse but no initial
segment of $M$ is.

$\Mladder$ is the least fully-iterable, sound ladder mouse, if it exists. If
$\Mladder$ exists it is obviously a minimal ladder mouse (and so a cofinal
ladder mouse), and $\Mladder$ projects to $\omega$ and so it is an $\omega$-mouse
in the sense of \cite{Proj_WO_In_Mod}.
\end{definition}

\begin{definition}
Let $M$ be a transitive model of a sufficient amount of set theory. 
\begin{itemize}
\item $M$ is 
$\Sigma^1_n$-\emph{correct} iff for all $\Sigma^1_n$ formula $\varphi$ and all
reals $x\in M$, $M\models\varphi[x]$ iff $\varphi(x)$. 
\item $M$ is \emph{projectively correct}
just in case it is $\Sigma^1_n$-correct for all $n\in\omega$.
\item $M$ is $\PiOneOmega$-\emph{correct} iff for all $\PiOneOmega$ codes $h$
and all reals $x\in M$, $M \models A^h(x)$ iff $A^h(x)$. Clearly
$M$ is $\PiOneOmega$-correct iff it is projectively correct.
\item $M$ is $\SigmaOneOmegaPlusOne$-\emph{correct} iff for all $\PiOneOmega$ codes $h$ for subsets of $\R^2$ and all reals $x\in M$,
$(\exists y \in \R\intersect M) M\models A^h(x,y)$ iff
$(\exists y \in \R) A^h(x,y)$.
\end{itemize}
\end{definition}

\begin{remarks} \
\begin{enumerate}
\item Because we are assuming there exists $\omega$ Woodin cardinals with a measurale cardinal above them, $\Mladder$ exists and is $(\omega_1+1)$-iterable.
\item For every $n\in\omega$, $\Mladder$ has a rank initial segment that satisfies
$\ZFC$+$\exists n$ Woodin cardinals.  By \cite{Proj_WO_In_Mod}, $\Mladder$ is projectively correct and so $\PiOneOmega$-correct.
\item Let $\sequence{\delta_n,\gamma_n}{n\in\omega}$ be a ladder over $\Mladder$.
For even $n$, $J^{\Mladder}_{\gamma_n}[g]$ is $\Sigma^1_{n+2}$-correct,
where $g$ is $\Coll(\delta_n,\omega)$-generic over $J^{\Mladder}_{\gamma_n}$.
\item Let $\sequence{\delta_n,\gamma_n}{n\in\omega}$ be a ladder over $M$.
$\delta_n$ is not necessarily fully Woodin in $M$, but $\delta_n$ is Woodin in $M$
with respect to functions in $J^M_{\gamma_n}$.
\item By the previous two items, we can, at least informally, think of
$\Mladder$ as the least mouse
$M$ such that for every $n\in\omega$ there is a cardinal $\delta$ of $M$ such
that $\delta$ is ``$\Sigma^1_n$-Woodin'' in $M$.
\end{enumerate}
\end{remarks}

We can now state the main theorem of the paper:

\begin{theorem}
Assume there exists $\omega$ Woodin cardinals with a measurable cardinal above them all.
Then $Q_{\omega+1} = \R \intersect \Mladder$.
\end{theorem}

We divide the above theorem up into its two directions:

\begin{theorem}
\label{MRealsAreDefinable}
Assume there exists $\omega$ Woodin cardinals with a measurable cardinal above them all.
Then $\R \intersect \Mladder \subseteq Q_{\omega+1}$.
\end{theorem}

\begin{theorem}
\label{DefinableRealsAreInM}
Assume there exists $\omega$ Woodin cardinals with a measurable cardinal above them all.
Then $Q_{\omega+1} \subseteq \Mladder$.
\end{theorem}

Theorem \ref{MRealsAreDefinable} was proven more than 25 years ago in
\cite{My_Thesis} and \cite{Mouse_Sets}. We give a sketch of the proof later
in Section \ref{section:beyond_first_pointclass}. The proof follows the same
line of reasoning as the proof that every real in $L$ is $\Delta^1_2$ in a
countable ordinal and the proof that every real in $M_1$ is $\Delta^1_3$ in
a countable ordinal. Namely we show that every initial segment of $\Mladder$
that projects to $\omega$ is $\PiOneOmega$ definable from its ordinal height.

As mentioned in the introduction, the main goal of this paper is to give Woodin's proof of Theorem \ref{DefinableRealsAreInM}.

Theorem \ref{DefinableRealsAreInM} will follow from a quasi-correctness theorem
for $\Mladder$. $\Mladder$ is $\PiOneOmega$-correct but not $\SigmaOneOmegaPlusOne$-correct.
Let $A\subset\R^2$ be $\PiOneOmega$ and let $x\in\R\intersect\Mladder$. It may
happen that there is a $y\in \R$ such that $A(x,y)$ but there is no such $y$
in $\Mladder$. But we will show that even if $\Mladder\not\models(\exists y) A(x,y)$,
$\Mladder$ can still determine whether or not the statement
``$(\exists y) A(x,y)$'' is true. This is similar to the situation with $M_1$ and
$\Sigma^1_3$. $M_1$ is not $\Sigma^1_3$-correct, but $M_1$ can still tell whether
or not $\Sigma^1_3$ statements are true. Let $A\subset\R^2$ be $\Pi^1_2$ and
let $x\in M_1$. Then $(\exists y\in\R)\, A(x,y)$ iff
$$M_1\models 1\force{\Coll(\delta,\omega)} (\exists y) A(x,y),$$
where $\delta$ is the Woodin of $M_1$.

Next we state the
$\SigmaOneOmegaPlusOne$-quasi-correctness for $\Mladder$ precisely and
we show how
Theorem \ref{DefinableRealsAreInM} follows from it.

If $z$ is a real, then by a \emph{mouse over $z$} we mean a mouse in the
sense of the theory of \cite{FSIT} augmented to use $z$ as an additional predicate.
(So this will be a model of the form $J_{\alpha}[\vec{E},z]$.)

\begin{theorem}
\label{QuasiCorrectness}
There is a formula $\psi$ in the language of Set Theory such that, for all
reals $z$, for all $M$ a countable, iterable, ladder-mouse over $z$,
there is an ordinal parameter $\theta<\omega_2^M$ such that
for all reals $x$ in $M$,
for all $\PiOneOmega$ codes $h$,
$$(\exists y) A^h(x,y) \Iff J^M_{\omega_2^M} \models \psi[\theta, h,x].$$
\end{theorem}

\begin{remark}
If $M=\Mladder$ then the ordinal parameter $\theta$ in Theorem \ref{QuasiCorrectness} can be eliminated.
We prove this in Section \ref{section:eliminate_ordinal_parameters}.
\end{remark}

Now, assuming Theorem \ref{QuasiCorrectness}, we prove Theorem \ref{DefinableRealsAreInM}.

\begin{proof}[proof of Theorem \ref{DefinableRealsAreInM}]
Let $x\in Q_{\omega+1}$. We need to show that $x\in\Mladder$.
Fix $\alpha<\omega_1$ such
that for all $w\in\WO$ with $|w|=\alpha$, $\singleton{x}$ is 
$\DeltaOneOmegaPlusOne(w)$.

Let $\Mprime$ be the $(\alpha+1)$-th iterate of $\Mladder$ by its least
measurable cardinal. ($\Mladder$ has many measurable cardinals because,
in the definition of a ladder over $M$, $\delta_n$ is Woodin in 
$J^M_{\gamma_n}$, so $J^M_{\gamma_n}\models$ ``$\delta_n$ is an inaccessible
limit of measurable cardinals'', but $\delta_n$ is also a cardinal of $M$
and so, by strong-acceptability, $\delta_n$ is a limit of measurable cardinals
in $M$.)
So $\Mprime$ is a countable, iterable ladder mouse with
a ladder whose $\delta_0$ is greater than $\alpha$. Let $g$ be $\Coll(\alpha,\omega)$-generic
over $\Mprime$. It suffices to see that $x\in\Mprime[g]$.

We can rearrange $\Mprime[g]$ as a mouse over $z$, where $z$ is some real coding $g$.
Let $M$ be this mouse over $z$. $M$ is a countable, iterable ladder-mouse over $z$ and there
is a real $w\in M\intersect\WO$ with $|w|=\alpha$.

Since $\singleton{x}$ is  $\SigmaOneOmegaPlusOne(w)$, it is also true that
$x$ is a $\SigmaOneOmegaPlusOne(w)$ subset of $\omega\times\omega$.
Fix a $\PiOneOmega$ code $h$ such that for all $i,j\in\omega$,
$$x(i)=j \Iff (\exists y\in\R)A^h(i,j,w,y).$$

Let $\psi$ and $\theta$ be the formula and parameter given by Theorem \ref{QuasiCorrectness}
for $M$. Then for all
$i,j\in\omega$,
$$x(i)=j \Iff (\exists y\in\R)A^h(i,j,w,y) \Iff J^M_{\omega_2^M} \models \psi[\theta,h,i,j,w].$$
So $x\in M$.
\end{proof}

We now develop the ideas that will allow us to prove theorem \ref{QuasiCorrectness}. We
start in the next section with a Suslin representation for $\PiOneOmega$ sets.

\section{A Suslin Representation for $\PiOneOmega$}
\label{section:suslinrep}

Let $\bdelta^1_{\omega} = \sup_n \bdelta^1_n$. In this section we show how to
express a $\PiOneOmega$ set as the projection
of a tree on $\omega\times \bdelta^1_{\omega}$.

\begin{definition}
Let $h$ be a $\PiOneOmega$ code for a subset of $\R$.
We here define $T^{h}$,  a natural tree  on $\omega\times\bdelta^1_{\omega}$ that projects to $A^h$.

Recall that $G^n\subset\omega\times\R$ is a universal $\Pi^1_{2n+1}$ set, uniformly in $n$.

For $n,e\in\omega$, let $G^n_e = \setof{x\in\R}{G^n(e, x)}$ and
let $\varphi^n_e = \sequence{\varphi^n_{e,i}}{i\in\omega}$ be a $\Pi^1_{2n+1}$ scale on $G^n_e$ uniformly in $n$ and $e$,
with each of the norms $\varphi^n_{e,i}$ being regular.
Let $<^n_{e,i}$ be the prewellorder on $\R$ associated with $\varphi^n_{e,i}$. So ``$x,y\in G^n_e$ and $x <^n_{e,i} y$'' is
$\Pi^1_{2n+1}$, uniformly in $n,e,i$, and $\varphi^n_{e,i}(x) = $ the rank of $x$ in $<^n_{e,i}$.
$\varphi^n_{e,i} : G^n_e \map \bdelta^1_{2n+1}$.

Let $T^n_e$ be the tree for the scale $\varphi^n_e$:
$$T^n_e = \setof{\left(x\restr k,\sequence{\varphi^n_{e,i}(x)}{i<k}\right)}{x\in G^n_e \text{ and } k\in\omega}.$$
Then $G^n_e=p[T^n_e]$. (This is part of what it means that 
$\sequence{\varphi^n_{e,i}}{i\in\omega}$ is a scale on $G^n_e$. )

Fix a recursive bijection between $\omega$ and $\omega\times\omega$, 
written as $i=\angles{i_0,i_1}$,
with the property that $i_0,i_1\leq i$ for $i\in\omega$.

Now fix a $\PiOneOmega$ code $h$ and we will define the tree $T^h$.
Let $(s,u) \in \omega^n\times(\bdelta^1_{\omega})^n$ for some $n\in\omega$. Then $(s,u)\in T^h$ iff there is an
$x\in\R$ extending $s$ such that $(\forall i<n)\, x\in G^{i}_{h(i)}$ and
$u(i)=\varphi^{i_0}_{h(i_0),i_1}(x)$. Intuitively, one can think of $T^h$ as building a pair $(x,f)$
where $x$ is a real and $f$ is a certificate verifying that $x\in p[T^n_{h(n)}]$ for all $n\in\omega$.

We say that $T$ is a $\PiOneOmega$-tree iff $T=T^h$ for some $\PiOneOmega$ code $h$.

Similarly we define $\PiOneOmega$-trees $T^h$ for $h$ a $\PiOneOmega$ code for a
subset of $\R^k$ for $k\in\omega$.
\end{definition}

Our coding of pairs of integers by integers induces a mapping from $\omega^{\omega}$ to $(\omega^{\omega})^{\omega}$: 
$f\mapsto\sequence{f_n}{n\in\omega}$ given by $f_n(i) = f(\angles{n,i})$.

\begin{proposition}
$A^h = p[T^h]$.
\end{proposition}
\begin{proof}
Let $x\in A^h = \Intersection{n} G^n_{h(n)}$. Let $f:\omega\to\bdelta^1_{\omega}$
be given by $f(i)=\varphi^{i_0}_{h(i_0),i_1}(x)$. Then $(x,f)\in [T^h]$.

Conversely, suppose $(x,f)\in [T^h]$. We claim that $\forall n \, (x,f_n)\in[T^n_{h(n)}]$.
Let $m > k > n \in \omega$
be such that $\angles{n,i}<m$ for all $i<k$. Let $\xprime$ be a real
such that $\xprime\restr m = x \restr m$ and $\xprime$ witness that
$(x\restr m, f\restr m)\in T^h$. Then $\xprime\in G^n_{h(n)}$
and $(\forall i < k)\, f_n(i)=f(\angles{n,i})=\varphi^{n}_{h(n),i}(\xprime)$.
So $\xprime$ witnesses that $(x\restr k, f_n\restr k)\in T^n_{h(n)}$.
As $n,k$ were arbitrary, $\forall n \, (x,f_n)\in[T^n_{h(n)}]$.

\end{proof}

\begin{remark}
\item The definition of the first $n$ levels of $T^h$ depends only on
 $\sequence{G^i_{h(i)}}{i<n}$ and
the norms $\sequence{\varphi^i_{h(i), j}}{i,j<n}$, all of which are $\Pi^1_{2n+1}$.
\end{remark}

Next we want to consider the version of $T^h$ inside of a countable, transitive model.

Let $M$ be a countable transitive model of $\ZFC$ that is $\Pi^1_{2n+1}$-correct.
Let $A\subset\R$ be $\Pi^1_{2n+1}$, and let $\varphi:A\map\bdelta^1_{2n+1}$ be a regular
$\Pi^1_{2n+1}$ norm on $A$ with associated pre-wellorder $<_\varphi$.
Then there is a natural way to interpret $\varphi$ in $M$.
$\varphi^M:A\intersect M \map (\bdelta^1_{2n+1})^M$
is given by $\varphi^M(x) = $ the rank of $x$ in $<_\varphi \intersect M$.

\begin{definition}
In the situation described in the previous paragraph, we define
$\sigma^M_{\varphi}:\range(\varphi^M)\map\bdelta^1_{2n+1}$ by
$\sigma^M_{\varphi}(\varphi^M(x))=\varphi(x)$ for all $x\in A \intersect M$.
This is well-defined and order preserving since for $x,y\in A\intersect M$,
$\varphi^M(x)=\varphi^M(y)$ iff $\varphi(x)=\varphi(y)$ and
$\varphi^M(x)<\varphi^M(y)$ iff $\varphi(x)<\varphi(y)$.
\end{definition}

$\sigma^M_{\varphi}$ is the inverse of the transitive collapse of
$\varphi[A\intersect M]$.

If $M$ is projectively correct then $\sigma^M_{\varphi}$ does not depend on the norm $\varphi$.

\begin{lemma}
Let $M$ be a countable, transitive, projectively correct model of $\ZFC$.
Let $A,B$ be two $\Pi^1_{2n+1}$ sets of reals and let $\varphi_A:A\to\bdelta^1_{2n+1}$
and $\varphi_B:B\to\bdelta^1_{2n+1}$ be $\Pi^1_{2n+1}$ norms on $A$ and $B$ respectively.
Let $\alpha < \min(\ran(\varphi^M_A),\ran(\varphi^M_B))$. Then
$\sigma^M_{\varphi_A}(\alpha) = \sigma^M_{\varphi_B}(\alpha)$.
\end{lemma}
\begin{proof}
If $x\in A$ and $y\in B$ and $\varphi^M_A(x) = \varphi^M_B(y)$, then
$\varphi_A(x)=\varphi_B(y)$. This is because,
by Theorem 3.3.2 of \cite{HarringtonKechris}, the relation
``$\varphi_A(x)=\varphi_B(y)$'' is $\Delta^1_{2n+3}$ and so absolute for $M$.
\end{proof}

\begin{definition}
\label{norm_embedding_def}
Let $M$ be a countable, transitive, projectively correct model of $\ZFC$.
Then $\sigma^M:(\bdelta^1_{\omega})^M\to\bdelta^1_{\omega}$ is the union of all
functions $\sigma^M_{\varphi}$ where $\varphi$ is a (lightface) projective norm on a
(lightface) projective set.

Let $h$ be a $\Pi^1_{\omega}$ code. Then we define
$\sigma^{M,h}:(T^h)^M\to T^h$ by
$$\sigma^{M,h}\left( (\angles{n_0,\cdots, n_k}, \angles{\alpha_0, \cdots, \alpha_k}  ) \right) =
(\angles{n_0,\cdots, n_k}, \angles{\sigma^M(\alpha_0), \cdots, \sigma^M(\alpha_k)}  ).$$
\end{definition}

\begin{lemma}
\label{tree_embedding_lemma}
Let $M$ be a countable, transitive, projectively correct model of $\ZFC$.
Let $h$ be a $\PiOneOmega$ code.
Then $\sigma^{M,h}:(T^h)^M\to T^h$ is a tree isomorphism
onto a subtree of $T^h$.
\end{lemma}
\begin{proof}
If $(s,u) = (\angles{n_0,\cdots, n_k}, \angles{\alpha_0, \cdots, \alpha_k}  ) \in (T^h)^M$
then there is a real $x\in M \intersect G^0_{h(0)} \intersect \cdots \intersect G^k_{h(k)}$ such
that $x(i)=n_i$ and $\alpha_i=(\varphi^{i_0}_{h(i_0),i_1}(x))^M$  for $0\leq i \leq k$.
So $\sigma^M(\alpha_i) = \varphi^{i_0}_{h(i_0),i_1}(x)$ for $0\leq i \leq k$ and so
$x$ witnesses that  $\sigma^{M,h}(s,u)\in T^h$.

That $\sigma^{M,h}$ is injective, preserves the lengths of finite sequences and sequence extension
is obvious.
\end{proof}

\begin{definition}
If $T$ is a tree on $\omega\times U$ and $x$ is a real then $T(x)$ is the tree through $x$:
$$T(x) = \setof{u}{(x\restr \lh(u), u) \in T}.$$

If $T$ is a $\PiOneOmega$ tree for a subset of $\R^2$ and $x\in\R$
we say that $T(x)$ is a $\PiOneOmega(x)$-tree
and a $\bPiOneOmega$-tree.
\end{definition}

\begin{remarks}
\label{remarks-about-correctness}
Let $M$ be a countable, transitive, projectively correct model of $\ZFC$.
$M$ is $\PiOneOmega$-correct, but not necessarily $\SigmaOneOmegaPlusOne$-correct.
Let $h$ be a $\PiOneOmega$ code for $A\subset\R^2$.
Let $x\in\R\intersect M$.
\begin{enumerate}
\item $\exists y A(x,y) \Iff T^h(x)$ is illfounded.
\item $M\models \exists y A(x,y) \Iff (T^h)^M(x)$ is illfounded.
\item $M\models \exists y A(x,y) \Implies \exists y A(x,y)$.
\item $\exists y A(x,y) \not\Implies M\models \exists y A(x,y)$.
\item $(y,f)\in[(T^h)^M] \Implies (y, \sigma^M[f]) \in [T^h]$.
\end{enumerate}
\end{remarks}

\begin{definition}
\label{correctly-wellfounded}
Let $M$ be a countable, transitive, projectively correct model of $\ZFC$. 
Let $h$ be a a $\PiOneOmega$ code and let $x\in \R\intersect M$.
Let $T=(T^h(x))^M$, so that $T\in M$ and $M\models ``T$ is a $\bPiOneOmega$ tree.''
Then we say that $(T,h,x,M)$ is \emph{correctly wellfounded} iff 
$(T^h(x))^V$ is wellfounded (and consequently $T = (T^h(x))^M$ is wellfounded.)
We say that $(T,h,x,M)$ is \emph{incorrectly wellfounded} iff $T$ is wellfounded but $(T^h(x))^V$ is illfounded. 

We will often write that $(T^h(x))^M$ is correctly or incorrectly wellfounded
as an abbreviation for the statement that 
$\left((T^h(x))^M,h,x,M\right)$ is.

We also extend these definitions to the case where $M$ itself is not a model of all of $\ZFC$ but
$M$ has a rank initial segment that is a model of $\ZFC$. Further we extend
these definitions to the case where $M$ is not wellfounded but has 
a rank initial segment in its wellfounded part that is a model of $\ZFC$.
\end{definition}

The key to our proof that $\Mladder$ is quasi-$\SigmaOneOmegaPlusOne$-correct is this:
We will show that if $T$ is a $\bPiOneOmega$-tree in $\Mladder$ then
$\Mladder$ is able to determine whether or not $T$ is correctly wellfounded.

\begin{theorem}
\label{CorrectBelowIncorrect}
Let $z$ be a real and $M$ a countable, iterable, ladder-mouse over $z$.
Suppose $T_1 = (T^{h_1}(x_1))^M$ and $T_2 = (T^{h_2}(x_2))^M$ are 
$\bPiOneOmega$-trees of $M$. Suppose $(T_1,h_1,x_1,M)$ is correctly
wellfounded and $(T_2,h_2,x_2,M)$ is incorrectly wellfounded. 
Then $\rank(T_1) < \rank(T_2)$.
\end{theorem}

\begin{remark}
\label{pure-correctly-wellfounded}
Let $M$ be as in the theorem and let $T$ be a $\bPiOneOmega$-tree of $M$.
It follows from the theorem that we can make sense of saying that
$T$ is correctly or incorrectly wellfounded without reference to an $h$ or $x$.
For if $T=(T^{h_1}(x_1))^M = (T^{h_2}(x_2))^M$, by the theorem it is not possible
for $(T,h_1,x_1,M)$ to be correctly wellfounded and $(T,h_2,x_2,M)$ to be
incorrectly wellfounded.
\end{remark}

Now, assuming Theorem \ref{CorrectBelowIncorrect}, we prove Theorem
\ref{QuasiCorrectness}.

\begin{proof}[proof of Theorem \ref{QuasiCorrectness}]
Let $z$ be a real and $M$ a countable, iterable, ladder-mouse over $z$.
Let $\theta$ be the minimum of the ranks of all
$\bPiOneOmega$ trees in $M$ that are  incorrectly wellfounded.
If there are no incorrectly wellfounded $\bPiOneOmega$ trees in $M$
then $M$ is $\Sigma^1_{\omega+1}$-correct and in this case let $\theta=0$.

$M$ is a model of $\PD$ and also is a strongly acceptable model of $\GCH$.
So $(\omega_1)^M<(\bdelta^1_{\omega})^M < (\omega_2)^M$.
So all $\bPiOneOmega$ trees in $M$ are in $J^M_{(\omega_2)^M}$ and either have a branch
in $J^M_{(\omega_2)^M}$ or have rank less than $(\omega_2)^M$.
So $\theta < (\omega_2)^M$.

Let $x\in\R\intersect M$ and let $h$ be a
$\PiOneOmega$ code. Then
$$(\exists y) A^h(x,y) \Iff J^M_{(\omega_2)^M}\models T^h(x) \text{ is illfounded} \,\OR\, (\theta>0 \,\AND\, \rank(T^h(x)) \geq \theta.)$$
\end{proof}

Previously we proved Theorem \ref{DefinableRealsAreInM} based on
Theorem \ref{QuasiCorrectness}. But it is also possible to give a proof
of Theorem \ref{DefinableRealsAreInM} based on Theorem \ref{CorrectBelowIncorrect}
without going through Theorem \ref{QuasiCorrectness} and thus avoiding
mentioning the parameter $\theta$.

\begin{proof}[alternate proof of Theorem \ref{DefinableRealsAreInM}]
Let $x\in Q_{\omega+1}$. We need to show that $x\in\Mladder$.

For $i\in\omega$ let $\ibar=\angles{i,0,0,\cdots}\in\R$. Thus we are representing integers as reals.
This allows us to avoid having to define what we mean by the tree through $i$, $T(i)$, when $T$ is a tree and $i$ an integer.

Fix a countable ordinal $\alpha$ such that $x$ is $\DeltaOneOmegaPlusOne$ in every representative
of $\alpha$.

Let $\Mprime$ be the $(\alpha+1)$-th iterate of $\Mladder$ by its least
measurable cardinal. So $\Mprime$ is a countable, iterable ladder mouse with
a ladder whose $\delta_0$ is greater than $\alpha$. Let $g$ be $\Coll(\alpha,\omega)$-generic
over $\Mprime$. It suffices to see that $x\in\Mprime[g]$.

We can rearrange $\Mprime[g]$ as a mouse over $z$, where $z$ is some real coding $g$.
Let $M$ be this mouse over $z$. $M$ is a countable, iterable ladder-mouse over $z$ and there
is a real $w\in M\intersect\WO$ with $|w|=\alpha$.

Fix two $\PiOneOmega$ codes $h_1,h_2$ such that for all $i,j\in\omega$,
$$x(i)=j \Iff (\exists y\in\R)A^{h_1}(\ibar,\jbar,w,y) \Iff (\forall y\in\R)\neg A^{h_2}(\ibar,\jbar,w,y).$$

\begin{claim}
For all $i,j\in\omega$,
$$x(i)=j \Iff  M \models T^{h_1}(\ibar,\jbar,w) \text{ is illfounded } \,\OR\, \rank(T^{h_2}(\ibar,\jbar,w)) < \rank(T^{h_1}(\ibar,\jbar,w)).$$
\end{claim}
\begin{subproof}[Proof of Claim.]
First suppose $x(i)=j$. Then $T^{h_1}(\ibar,\jbar,w)$ is illfounded and 
$T^{h_2}(\ibar,\jbar,w)$ is well-founded. By
Remarks \ref{remarks-about-correctness}, $(T^{h_2}(\ibar,\jbar,w))^M$
is correctly-wellfounded. If $M \not\models T^{h_1}(\ibar,\jbar,w) \text{ is illfounded }$, then $(T^{h_1}(\ibar,\jbar,w))^M$ is incorrectly-wellfounded
and so by Theorem 
\ref{CorrectBelowIncorrect}, 
$\rank(T^{h_2}(\ibar,\jbar,w))^M < \rank(T^{h_1}(\ibar,\jbar,w))^M$.

Conversely, suppose $x(i)\not=j$. Then $T^{h_1}(\ibar,\jbar,w)$ is well-founded and 
$T^{h_2}(\ibar,\jbar,w)$ is illfounded. By Remarks \ref{remarks-about-correctness},
$(T^{h_1}(\ibar,\jbar,w))^M$ is correctly well-founded and 
$(T^{h_2}(\ibar,\jbar,w))^M$ is either illfounded or incorrectly well-founded.
By Theorem \ref{CorrectBelowIncorrect}, 
$M\not\models \rank(T^{h_2}(\ibar,\jbar,w)) < \rank(T^{h_1}(\ibar,\jbar,w))$.
\end{subproof}

So $x\in M$.

\end{proof}

In the next section we turn to the proof of Theorem \ref{CorrectBelowIncorrect}.

\section{Main Proof}
\label{section:mainproof}

This section is devoted to a proof of Theorem \ref{CorrectBelowIncorrect}.

Let $z$ be a real and $M$ a countable, iterable, ladder-mouse over $z$.
Suppose $M\models ``T_1$ and $T_2$ are $\bPiOneOmega$-trees'' and
suppose $T_1$ is correctly wellfounded and
$T_2$ is incorrectly wellfounded. We want to prove that
$\rank(T_1) < \rank(T_2)$. To simplify the notation in this section we will
assume $z=0$. The proof
with arbitrary $z$ is the same, but relativized to $z$.

Theorem \ref{CorrectBelowIncorrect} will follow from the three Lemmas below:
Lemmas \ref{generic_over_all_windows}, \ref{hjorthlemma} and \ref{m_star_existence}.

\begin{lemma}
\label{generic_over_all_windows}
Let $M$ be a countable, iterable, ladder-mouse
with $\sequence{\delta_n,\gamma_n}{n\in\omega}$ a ladder over $M$.
Let $y\in\R$.
Then there is a non-dropping iteration $\pi:M\to \Mprime$ such that,
letting $\sequence{\deltaprime_n,\gammaprime_n}{n\in\omega}$ be the image under $\pi$
of the ladder, we have that $\Mprime$ is a countable, iterable ladder-mouse
with $\sequence{\deltaprime_n,\gammaprime_n}{n\in\omega}$ a ladder over $\Mprime$
and for each $n\in\omega$ there is a $g_n$ such that $g_n$
is $\Coll(\deltaprime_n,\omega)$-generic over
$J^{\Mprime}_{\gammaprime_n}$ and $y\in J^{\Mprime}_{\gammaprime_n}[g_n]$.
\end{lemma}

The iteration described in the previous lemma will come from performing a
Woodin genericity iteration over each of the levels of the ladder sequentially.

\begin{lemma}
\label{hjorthlemma}
Let $\Gamma = \Sigma^1_n$ for some $n>0$.
Let $M$ be a countable, transitive model of $\ZFC$ and  $\delta>\omega$ a cardinal of $M$.
Suppose that whenever $\P\in V^M_{\delta+1}$ is a poset and $g$ is $\P$-generic over $M$, then
$M[g]$ is $\Gamma$-correct (in $V$). Let $g_1$ and $g_2$ be $\Coll(\delta,\omega)$-generic over $M$.
Let $E\subset\R^2$ be a $\Gamma$, thin equivalence relation.
Then $M[g_1]$ and $M[g_2]$ meet the same $E$-equivalence classes.
\end{lemma}

\begin{note} In the above Lemma we could replace $\Sigma^1_n$ with any
pointclass $\Gamma$ for which we can make sense of the notion of
$\Gamma$-correctness.
\end{note}

We will use Lemma \ref{hjorthlemma} via the following corollary.

\begin{corollary}
\label{hjorthcorollary} Fix $n\in\omega$.
Let $M$ be a countable, transitive model of $\ZFC$ and  $\delta>\omega$ a cardinal of $M$.
Suppose that whenever $\P\in V^M_{\delta+1}$ is a poset and $g$ is $\P$-generic over $M$, then
$M[g]$ is $\Sigma^1_{2n+4}$-correct. Let $g_1$ and $g_2$ be $\Coll(\delta,\omega)$-generic over $M$.

Let $h$ be a $\PiOneOmega$ code, let $y\in M[g_1]$ and suppose that $y\in \Intersection{i<n} G^i_{h(i)}$. Then there is
a $\yprime \in M[g_2]$ such that
\begin{itemize}
\item $\yprime \restr n = y \restr n$
\item $\yprime\in \Intersection{i<n} G^i_{h(i)}$.
\item For $i<n$, $\varphi^{i_0}_{h(i_0),i_1}(\yprime) = \varphi^{i_0}_{h(i_0),i_1}(y)$
\end{itemize}
\end{corollary}
\begin{proof}
Consider the binary relation $E$ on $\R^2$ given by
\begin{align*}
E(z,w) \Iff &z,w\notin \Intersection{i<n} G^i_{h(i)} \OR \Bigg( z,w \in \Intersection{i<n} G^i_{h(i)} \, \AND  \\
&\left( z\restr n, \sequence{\varphi^{i_0}_{h(i_0),i_1}(z)}{i<n}\right) = \left( w\restr n, \sequence{\varphi^{i_0}_{h(i_0),i_1}(w)}{i<n}\right) \Bigg).
\end{align*}

Clearly $E$ is a $\Sigma^1_{2n+4}$ equivalence relation on $\R$. ($\Sigma^1_{2n+4}$ is more than we
need---we are not claiming it is optimal.)

We claim that $E$ is thin. To see this, suppose towards a contradiction that $S$ is a perfect set of $E$-inequivalent reals.
There is a natural wellorder $<$ on $S$ defined by

\begin{align*}
z<w \Iff &z\in \Intersection{i<n} G^i_{h(i)} \AND \Bigg( w \notin \Intersection{i<n} G^i_{h(i)} \, \OR   \bigg[w \in \Intersection{i<n} G^i_{h(i)} \, \AND\\
&\left( z\restr n, \sequence{\varphi^{i_0}_{h(i_0),i_1}(z)}{i<n}\right) \lexless \left( w\restr n, \sequence{\varphi^{i_0}_{h(i_0),i_1}(w)}{i<n}\right) \bigg]\Bigg).
\end{align*}

Since $<$ is projectively definable and $S$ is perfect, this gives a projectively definable wellorder of $\R$. But that
contradicts $\PD$. So $E$ is thin.

By Lemma \ref{hjorthlemma} with $\Gamma=\Sigma^1_{2n+4}$, there is a
$\yprime\in M[g_2]$ such that $E(y_1, y_2)$.
\end{proof}

\begin{lemma}
\label{m_star_existence}
Let $M$ be a countable, iterable ladder-mouse
with $\sequence{\delta_n,\gamma_n}{n\in\omega}$ a ladder over $M$ and
$\gamma=\sup\setof{\gamma_n}{n\in\omega}$. Then there is a countable,
not-necessarily-wellfounded model $\Mstar$ and a $\Sigma_0$-elementary-embedding
$$\pi:M\to\Mstar$$
such that
\begin{itemize}
\item $\crit(\pi) = (\omega_1)^M$.
\item $\pi((\omega_1)^M) = \gamma \in \wfp(\Mstar)$.
\item $\Mstar$ is projectively correct.
\item $J^{M}_{\gamma_n}\in\Mstar$, for each $n$.
\end{itemize}
\end{lemma}

The embedding described in the previous lemma will come from taking a stationary tower
ultrapower.

Now, assuming Lemmas \ref{generic_over_all_windows}, \ref{hjorthlemma} and
\ref{m_star_existence}, we will give a proof of Theorem \ref{CorrectBelowIncorrect}.

\begin{proof}[proof of Theorem \ref{CorrectBelowIncorrect}]
As mentioned at the top of the section, we will assume that the real parameter
$z=0$ to simplify the notation.
Let $M$ be a countable, iterable ladder-mouse.

Suppose $T_1 = (T^{h_1}(x_1))^M$ and $T_2 = (T^{h_2}(x_2))^M$ are 
$\bPiOneOmega$-trees of $M$. Suppose $(T_1,h_1,x_1,M)$ is correctly
wellfounded and $(T_2,h_2,x_2,M)$ is incorrectly wellfounded. 
We want to prove that $\rank(T_1) < \rank(T_2)$.

To simplify the notation again, we will assume that $x_1=x_2=0$.

By assumption, in $V$, $T^{h_1}$ is wellfounded and $T^{h_2}$ is illfounded. Let
$y\in p[T^{h_2}] = A^{h_2} = \Intersection{n} G^n_{h_2(n)}$.
So then
$$\left(y,\sequence{\varphi^{n_0}_{h_2(n_0),n_1}(y)}{n\in\omega} \right)\in[T^{h_2}].$$

Let $\sequence{\delta_n,\gamma_n}{n\in\omega}$ be a ladder over $M$.
By applying Lemma \ref{generic_over_all_windows} and replacing $M$ with the $\Mprime$ of
that lemma, we may assume that we have, for each $n$, $g_n$
$\Coll(\delta_n,\omega)$-generic over $J^M_{\gamma_n}$ such that $y\in J^M_{\gamma_n}[g_n]$.

Now we apply Lemma \ref{m_star_existence} to obtain $\pi:M\to\Mstar$ as in
the lemma.
Let $T_1^* = \pi(T_1), T_2^* = \pi(T_2)$.
It suffices to see that $\Mstar\models\text{``}\rank(T_1^*) < \rank(T_2^*)$''.

Since $\Mstar$ is projectively correct, it follows that $(\bdelta^1_n)^{M^*}\in\wfp(M^*)$
for each $n$, and so $(\bdelta^1_{\omega})^{\Mstar}\in\wfp(\Mstar)$. So
$T_1^*,T_2^*\in\wfp(\Mstar)$ and $T_1^* = (T^{h_1})^{\Mstar}$
and $T_2^* = (T^{h_2})^{\Mstar}$.

In Definition \ref{norm_embedding_def} and
Lemma \ref{tree_embedding_lemma} we assumed that $M$ was a countable, transitive, projectively correct model of $\ZFC$. But the definition and lemma work also for $\Mstar$ even though $\Mstar$ is not
necessarily wellfounded and $\Mstar$ is not necesarily a model of $\ZFC$, but rather it has
rank initial segments that are. Thus we have
$\sigma^{\Mstar}:(\bdelta^1_{\omega})^{\Mstar}\to\bdelta^1_{\omega}$, and
we have that the extensions
$\sigma^{\Mstar,h_1}:T_1^*\map T^{h_1}$ and $\sigma^{\Mstar,h_2}:T_2^*\map T^{h_2}$ are tree
embeddings.

$\Mstar\models \text{``}T_1^* \text{ and } T_2^*$ are wellfounded''. But $\Mstar$ is not
necessarily wellfounded so it doesn't necessarily follow that $T_1^*$ and $T_2^*$ are really wellfounded.
Now $T^{h_1}$ is wellfounded and since we have $\sigma^{\Mstar,h_1}:T_1^*\map T^{h_1}$,
we \emph{do} know that $T_1^*$ is wellfounded. This implies that
$\rank(T_1^*)\in\wfp(\Mstar)$.
To see that $\Mstar\models\text{``}\rank(T_1^*) < \rank(T_2^*)$'' it suffices to
see that $\rank(T_2^*)\notin\wfp(\Mstar)$, or equivalently that $T_2^*$ is not wellfounded.
We will show that $y\in p[T_2^*]$. (Thus we are also showing that in the present
circumstance, $M^*$ is necessarily illfounded.)
This follows from the following

\begin{claim}
For all $n>0$, there is a $y_n\in \R\intersect\Mstar$ such that $y_n\restr n = y \restr n$,
$y_n \in  \Intersection{i<n} G^i_{h_2(i)}$,
and for all $i < n$, $\varphi^{i_0}_{h_2(i_0),i_1}(y_n) = \varphi^{i_0}_{h_2(i_0),i_1}(y)$.
\end{claim}

Assuming the claim, $\varphi^{i_0}_{h_2(i_0),i_1}(y)\in\ran(\sigma^{\Mstar})$ for all $i$
and letting $\alpha_i$ be such that $\sigma^{\Mstar}(\alpha_i)=\varphi^{i_0}_{h_2(i_0),i_1}(y)$
we have that for all $n$,
$$\left(y\restr n, \sequence{\alpha_i}{i<n} \right)=
\left(y_n\restr n, \sequence{(\varphi^{i_0}_{h_2(i_0),i_1}(y_n))^{M^*}}{i<n} \right)\in[T^*_2]$$
and so
$$\left(y, \sequence{\alpha_i}{i\in\omega} \right)\in[T^*_2].$$

To finish the proof of Theorem \ref{CorrectBelowIncorrect} we give the

\begin{subproof}[proof of the claim]
Fix $n$ and let $N=J^M_{\gamma_{2n+2}}$. $N$ has $2n+2$ Woodin cardinals above $\delta_{2n+2}$,
so if $\P\in N$ is a poset with $\P\subset V^N_{\delta_{2n+2}}$ and $g$ is $\P$-generic over $N$, then $N[g]$
is $\Sigma^1_{2n+4}$-correct.

Since $N$ is countable in $M^*$, there is a $g\in\Mstar$ that is
$\Coll(\delta_{2n+2},\omega)$-generic over $N$. By Corollary \ref{hjorthcorollary}
there is a $\yprime\in N[g]$ such that
$\yprime\restr n = y \restr n$,
$\yprime \in  \Intersection{i<n} G^i_{h_2(i)}$,
and for all $i < n$, $\varphi^{i_0}_{h_2(i_0),i_1}(\yprime) = \varphi^{i_0}_{h_2(i_0),i_1}(y)$.
Let $y_n=\yprime$.

\end{subproof}

And that completes the proof of Theorem \ref{CorrectBelowIncorrect}.

\end{proof}

In the remainder of this section we discharge our obligation to give the proofs of
Lemmas \ref{generic_over_all_windows}, \ref{hjorthlemma} and
\ref{m_star_existence}.

\begin{proof}[sketch of proof of Lemma \ref{generic_over_all_windows}]
Lemma \ref{generic_over_all_windows} is not a new result: it was essentially
proven in Section 4 of \cite{Mouse_Sets}. We give only a sketch here.

Since $\delta_0$ is Woodin in $J^M_{\gamma_0}$,
we can perform the usual genericity iteration on $J^M_{\gamma_0}$
(see Theorem 7.14 from \cite{OutlineOfInnerModelTheory}), yielding
$\pibar_0:J^M_{\gamma_0}\to P_1$ such that
there is a $g$, $\Coll(\pibar_0(\delta_0),\omega)$-generic over $P_1$ such that $y\in P_1[g]$.

Since $\delta_0$ a cardinal in $M$, a total extender over $J^M_{\delta_0}$ is total
over $M$. A similar fact is true about the image of $\delta_0$ under non-dropping
iterations of $M$.
Therefore the $\pibar_0$-iteration of $J^M_{\gamma_0}$ induces a corresponding iteration on all of $M$
yielding a non-dropping iteration $\pi_0:M\to \Mprime_1$ such that
letting $\sequence{\delta^{(1)}_n,\gamma^{(1)}_n}{n\in\omega}$ be the image under $\pi_0$
of the ladder of $M$, we have that $\Mprime_1$ is a countable, iterable ladder-mouse
with $\sequence{\delta^{(1)}_n,\gamma^{(1)}_n}{n\in\omega}$ a ladder over $\Mprime_1$
and there is a $g_0$ such that $g_0$
is $\Coll(\delta^{(1)}_0,\omega)$-generic over
$J^{\Mprime_1}_{\gamma^{(1)}_0}$ and $y\in J^{\Mprime_1}_{\gamma^{(1)}_0}[g_0]$.

Now we repeat the argument of the previous two paragraphs using $\Mprime_1$ instead of $M$
and using $\delta^{(1)}_1$ in place of $\delta_0$. This yields a non-dropping iteration
$\pi_1:\Mprime_1\to\Mprime_2$. We can take the critical point of $\pi_1$ to be
greater than $\gamma^{(1)}_0$. Then,
letting $\sequence{\delta^{(2)}_n,\gamma^{(2)}_n}{n\in\omega}$ be the image under $\pi_1$
of the ladder of $\Mprime_1$, we have that $\Mprime_2$ is a countable, iterable ladder-mouse
with $\sequence{\delta^{(2)}_n,\gamma^{(2)}_n}{n\in\omega}$ a ladder over $\Mprime_2$
and for $n=0,1$ there is a $g_n$ such that $g_n$
is $\Coll(\delta^{(2)}_n,\omega)$-generic over
$J^{\Mprime_2}_{\gamma^{(2)}_n}$ and $y\in J^{\Mprime_2}_{\gamma^{(2)}_n}[g_n]$.

We continue this process yielding $\pi_n:\Mprime_n\to\Mprime_{n+1}$ and
let $\pi:M\to\Mprime$ be the direct limit.

\end{proof}

Next we turn to the proof of Lemma \ref{hjorthlemma}.
The proof is based on similar arguments by Greg Hjorth (see Lemma 2.2 from \cite{Hjorth_applications_of_course_imt}) and Foreman and Magidor (see Theorem 3.4 of \cite{Foreman_and_Magidor}.)
See also Lemma 3.1 from \cite{Schlict_Thin_Equivalence_Relations}.

\begin{proof}[proof of Lemma \ref{hjorthlemma}]
Let $\Gamma = \Sigma^1_n$ for some $n>0$.
Let $M$ be a countable, transitive model of $\ZFC$ and  $\delta>\omega$ a cardinal of $M$.
Suppose that whenever $\P\in V^M_{\delta+1}$ is a poset and $g$ is $\P$-generic over $M$, then
$M[g]$ is $\Gamma$-correct. Let $g_1$ and $g_2$ be $\Coll(\delta,\omega)$-generic over $M$.
Let $E\subset\R^2$ be a $\Gamma$, thin equivalence relation.
We wish to prove that $M[g_1]$ and $M[g_2]$ meet the same $E$-equivalence classes.

We may assume that $g_1\times g_2$ is $\Coll(\delta,\omega)\times \Coll(\delta,\omega)$-generic over $M$.
This is because there is an h such that $h$ is $\Coll(\delta,\omega)$-generic over both
both $M[g_1]$ and $M[g_2]$. If $M[g_1]$ and $M[h]$ meet the same $E$-equivalence classes and
$M[g_2]$ and $M[h]$ meet the same $E$-equivalence classes, then
$M[g_1]$ and $M[g_2]$ meet the same $E$-equivalence classes.

So suppose towards a contradiction that $g_1\times g_2$ is $\Coll(\delta,\omega)\times \Coll(\delta,\omega)$-generic over $M$
and there is an $x\in\R\intersect M[g_1]$ such that for all $y\in \R\intersect M[g_2]$, $\neg E(x,y)$.

\begin{claim}[Claim 1]
There is a name for a real $\tau\in M$ such that for all $h_1,h_2$
with $h_1\times h_2$ $\Coll(\delta,\omega)\times \Coll(\delta,\omega)$-generic over $M$,
$\neg E(\tau^{h_1},\tau^{h_2})$.
\end{claim}
\begin{subproof}[proof of Claim 1]
Let $\tau\in M$ be such that $x=\tau^{g_1}$. By hypothesis $M[g_1\times g_2]$ is $\Gamma$-correct
so $M[g_1\times g_2]\models \neg E(\tau^{g_1},\tau^{g_2})$.

Let $\dot{g}_{\text{left}}$ and $\dot{g}_{\text{right}}$
be the canonical $\Coll(\delta,\omega)\times \Coll(\delta,\omega)$-names in $M$ for the left and right component of the generic.
Let $\check{\tau}$ be the canonical $\Coll(\delta,\omega)\times \Coll(\delta,\omega)$-name for $\tau\in M$.
Then $\check{\tau}$, $\dot{g}_{\text{left}}$ and $\dot{g}_{\text{right}}$ are all invariant under automorphisms of $\Coll(\delta,\omega)\times \Coll(\delta,\omega)$.
Since $\Coll(\delta,\omega)\times \Coll(\delta,\omega)$ is homogeneous
$$1\forcein{\Coll(\delta,\omega)\times \Coll(\delta,\omega)}{M} \neg E(\check{\tau}^{\dot{g}_{\text{left}}},\check{\tau}^{\dot{g}_{\text{right}}}).$$
Let $h_1,h_2$ be such that $h_1\times h_2$ is $\Coll(\delta,\omega)\times \Coll(\delta,\omega)$-generic over $M$.
Then $M[h\times h_2]\models \neg E(\tau^{h_1},\tau^{h_2})$ and $M[h\times h_2]$ is $\Gamma$-correct so $\neg E(\tau^{h_1},\tau^{h_2})$.
\end{subproof}

To obtain a contradiction and prove the lemma we will show

\begin{claim}[Claim 2]
There is a set $H$ such that for all $h_1,h_2\in H$ with $h_1\not= h_2$,
$h_1\times h_2$ is $\Coll(\delta,\omega)\times \Coll(\delta,\omega)$-generic over $M$
and $\setof{\tau^{h}}{h\in H}$ is a perfect set of reals.
\end{claim}
\begin{subproof}[proof of Claim 2]
Let $\sequence{C_n}{n=1,2,\cdots}$ enumerate the dense open subsets of $\Coll(\delta,\omega)\times \Coll(\delta,\omega)$ in $M$
and let $D_n=\Intersection{i\leq n} C_n$.
For $s\in 2^{<\omega}$ we define $p_s\in \Coll(\delta,\omega)$ such that
\begin{enumerate}
\item $s\subset t \Implies p_s \subset p_t$
\item If $|s| = |t| = n$ and $s\not=t$, then $(p_s,p_t) \in D_n$
\item If $|s| = n$ then $p_s$ decides $\tau\restr n$.
\end{enumerate}
It is easy to see that $p_s$ satisfying these conditions can be defined by induction on the length of $s$.
[Let $p_{\emptyset} = \emptyset$. Let $n\geq 1$ and suppose we have defined $p_s$ for all $s$ of length $n-1$.
Let $\sequence{(s,t)_i}{i=1\cdots k}$ enumerate all pairs $(s,t)\in{\pre{2}{n}\times \pre{2}{n}}$ such that $s\not=t$.
Let $p^0_s = p_{s\upharpoonright (n-1)}$ for all $s\in \pre{2}{n}$. For $i=1\cdots k$, let $(s,t)=(s,t)_i$ and let $p^i_s$ and $p^i_t$ be
chosen so that $p^{i-1}_s\subset p^i_s$ and $p^{i-1}_t\subset p^i_t$ and $(p^i_s,p^i_t)\in D_{n}$. For all $u\in \pre{2}{n}$ with
$u\not=s$ and $u\not=t$, let $p^i_u=p^{i-1}_u$.
Then $\setof{p^k_s}{s\in\pre{2}{n}}$ maintains induction hypotheses (1) and (2). For each $s$, choose $p_s$ extending $p^k_s$ to also
maintain induction hypothesis (3).]

For $x\in 2^{\omega}$ let $h_x$ be the filter on $\Coll(\delta,\omega)$ generated by $\setof{p_{x\upharpoonright n}}{n\in\omega}$
and let $H=\setof{h_x}{x\in 2^{\omega}}$. If $x,y\in 2^{\omega}$ and $x\not=y$ then $h_x\times h_y$ is
$\Coll(\delta,\omega)\times \Coll(\delta,\omega)$-generic over $M$, because for sufficiently large $n$,
$x\restr n \not= y\restr n$ so that $(p_{x\upharpoonright n },p_{y\upharpoonright n })\in D_n$, and $p_{x\upharpoonright n }\in h_x$ and
$p_{y\upharpoonright n }\in h_y$.

Sine $x\to \tau^{h_x}$ is a continuous injection from Cantor space, $\setof{\tau^{h_x}}{x\in 2^{\omega}}$ is perfect.
\end{subproof}

By Claims 1 and 2 there is a perfect set of $E$-inequivalent reals, contradicting the fact that $E$ is thin.

\end{proof}

Finally we turn to the proof of Lemma \ref{m_star_existence}. This proof uses
stationary tower Forcing. See \cite{Larson_Book} and page 653, Section 34 of \cite{Jech_Book2} for background.

Let $M$ be a countable ladder-mouse
with $\sequence{\delta_n,\gamma_n}{n\in\omega}$ a ladder over $M$ and
$\gamma=\sup\setof{\gamma_n}{n\in\omega}$.

For $n<\omega$, let
$$\Q_{<\delta_n}^M=\setof{a\in J^M_{\delta_n}}{M\models a\subset \Powerset_{\omega_1}(J^M_{\delta_n}) \text{ and } a \text{ is stationary.}}$$
For $a,b\in \Q_{<\delta_n}^M$, say that $b\leq a$ iff $\union a \subseteq \union b$ and $\bar{b} \subseteq a$ where $\bar{b}=\setof{x\intersect\union a}{x\in b}$.
$\big(\Q_{<\delta_n}^M, <\big)$ is the countable stationary tower forcing up to $\delta_n$ in $M$. Also let $\Q_{<\gamma}^M=\Union{n<\omega}\Q_{<\delta_n}^M$.

\begin{lemma}
\label{LadderGeneric}
Let $M$ be a countable ladder-mouse
with $\sequence{\delta_n,\gamma_n}{n\in\omega}$ a ladder over $M$ and
$\gamma=\sup\setof{\gamma_n}{n\in\omega}$.
There is a $G\subset \Q_{<\gamma}^M$ such that for $n\in\omega$, $G\intersect J^M_{\delta_n}$ is $\Q_{<\delta_n}^M$-generic over $J^M_{\gamma_n}$.
\end{lemma}

First we show how Lemma \ref{m_star_existence} follows from Lemma
\ref{LadderGeneric}

\begin{proof}[proof of Lemma \ref{m_star_existence}.]
Let $M$ be a countable, iterable ladder-mouse
with $\sequence{\delta_n,\gamma_n}{n\in\omega}$ a ladder over $M$ and
$\gamma=\sup\setof{\gamma_n}{n\in\omega}$. We need to construct a countable,
not-necessarily-wellfounded model $\Mstar$ and a $\Sigma_0$-elementary-embedding
$$\pi:M\to\Mstar$$
such that
\begin{itemize}
\item $\crit(\pi) = \omega_1^M$.
\item $\pi(\omega_1^M) = \gamma \in \wfp(\Mstar)$.
\item $\Mstar$ is projectively correct.
\item $J^{M}_{\gamma_n}\in\Mstar$, for each $n$.
\end{itemize}

$\pi:M\to\Mstar$ will be a stationary tower ultrapower.

Let $G$ be as in Lemma \ref{LadderGeneric}.

If $A$ is an infinite set in $J^M_{\gamma}$, let
$$G\restr A = G \intersect  \Powerset(\Powerset_{\omega_1}(A))^M = \setof{a\in G}{\union a = A}.$$

$G\restr A$ is an $M$-ultrafilter on $\Powerset_{\omega_1}(A)^M$.

We form the ultrapower $\Ult(M,G)$ by taking all functions $f\in J^M_{\gamma}$ with $\dom(f)=\Powerset_{\omega_1}(A)^M$ for some infinite set $A\in J^M_{\gamma}$.
If $\dom(f)=\Powerset_{\omega_1}(A)^M$ and  $\dom(g)=\Powerset_{\omega_1}(B)^M$ we say that $f\sim g$ iff
$$\setof{x\in \Powerset_{\omega_1}(A\union B)^M}{f(x\intersect A) = g(x\intersect B)} \in G$$
and we say that $[f] E [g]$ iff
$$\setof{x\in \Powerset_{\omega_1}(A\union B)^M}{f(x\intersect A) \in g(x\intersect B)} \in G.$$
We have $\Sigma_0$ \L{}os's Theorem for $\Ult(M,G)$.
We let $\Mstar=\Ult(M,G)$, identifying the well-founded part of $\Mstar$ with the transitive set to which it is isomorphic.

Letting $\pi:M\to\Mstar$ be given by $\pi(x)=[c_x]$, where $c_x$ is the constant $x$ function, we have that $\pi$ is $\Sigma_0$-elementary.

If $A$ is any infinite set in $J^M_{\gamma}$, the identity function on $\Powerset_{\omega_1}(A)^M$ represents $\pi[A]$ and so $A\in \Mstar$.
Thus $J^{M}_{\gamma_n}\in\Mstar$, for each $n$ and $\gamma\subseteq\wfp(\Mstar)$.

If $\alpha<\gamma$ is any ordinal, the function $f$ with domain $\Powerset_{\omega_1}(\alpha)^M$ given by $f(x)=\ot(x)$ represents $\alpha$. It
follows that $\alpha$ is countable in $\Mstar$. So $\crit(\pi)=\omega_1^M$ and $\pi(\omega_1^M)\geq \gamma$.

\begin{claim}[Claim 1]
$\pi(\omega_1^M)=\gamma$
\end{claim}
\begin{subproof}
Let $f$ represent a countable ordinal in $\Ult(M,G)$. We will show that $[f]_G < \gamma$. Let $A=\union\dom(f)$ so that
$f:\Powerset_{\omega_1}^M(A)\to\omega_1^M$. We will find a $\kappa < \gamma$ such that
$f\in J^M_{\kappa}$, and such that
$$\setof{x\in\Powerset_{\omega_1}^M(J^M_{\kappa})}{f(x\intersect A)\leq \ot(x\intersect \kappa)}\in G.$$
Since the function $g$ given by $g(x)=\ot(x\intersect \kappa)$ represents $\kappa$ in $\Ult(M,G)$,
we will have that $[f]_G \leq \kappa$.

Fix $n\in\omega$ so that $A,f,a\in J^M_{\delta_n}$. Work in $J^M_{\gamma_n}$, where $\delta_n$ is Woodin.
In particular, $\delta_n$ is a limit of completely J\'{o}nsson cardinals.

Fix an arbitrary $a\in \Q^M_{\delta_n}$.

Let $\kappa<\delta_n$ be such
that $\kappa$ is completely J\'{o}nsson in $J^M_{\gamma_n}$ and $A,f,a\in J^M_{\kappa}$. Let $b$ be the set of
countable $X\subset J^M_{\kappa}$ such that
\begin{itemize}
\item $X\intersect (\union a) \in a$
\item $f(X\intersect A) \leq \ot(X\intersect \kappa)$.
\end{itemize}
The proof of Theorem 2.7.8 of \cite{Larson_Book} shows that $b$ is stationary.
Thus $b\leq a$.

Since $a$ was arbitrary in $\Q^M_{\delta_n}$ and $G\intersect J^M_{\delta_n}$ is $J^M_{\gamma_n}$-generic,
there is such a $\kappa$ and $b$ with $b\in G\intersect J^M_{\delta_n}$. For this
$\kappa$, $[f]_G \leq \kappa$.
\end{subproof}

In particular, $\gamma$ is definable in $\Mstar$ and so $\gamma\in\wfp(\Mstar)$.

\begin{claim}[Claim 2]
$\Mstar$ is projectively correct.
\end{claim}
\begin{subproof}
Let $x=[f]_G$ be a real in $\Ult(M,G)$ with $f\in J^M_{\gamma}$.
 Let $A=\union\dom(f)$ so that
$f:\Powerset_{\omega_1}^M(A)\to\R$. Let $n$ be such that $f\in J^M_{\delta_n}$.
Then $f$ also represents $x$ in
$\Ult(J^M_{\gamma_n},G\intersect J^M_{\delta_n})$.

Since $\delta_n$ is Woodin in $J^M_{\gamma_n}$,
$$\R\intersect\Ult(J^M_{\gamma_n},G\intersect J^M_{\delta_n}) = \R\intersect J^M_{\gamma_n}[G\intersect J^M_{\delta_n}].$$

So we have that
$$\R\intersect \Mstar =  \Union{n} \R\intersect J^M_{\gamma_n}[G\intersect J^M_{\delta_n}].$$

Since $M$ is iterable, $J^M_{\gamma_{2n}}[G\intersect J^M_{\delta_{2n}}]$ is
$\Sigma^1_{2n+1}$-correct.

Since correctness is a closure property, it follows that
$\R\intersect\Mstar$ is projectively correct.
\end{subproof}

This completes the proof of Lemma \ref{m_star_existence} assuming Lemma
\ref{LadderGeneric}.

\end{proof}

Finally we give the proof of Lemma \ref{LadderGeneric}, which completes the
proof of Lemma \ref{m_star_existence}, and thus the proof of
Theorem \ref{CorrectBelowIncorrect}, and thus the proof of
Theorem \ref{QuasiCorrectness} and thus the proof of Theorem \ref{DefinableRealsAreInM}.

Here we will use more details about stationary tower forcing.
For the reader's convenience we copy here some of the definitions and results from 
\cite{Larson_Book} and Section 34 of \cite{Jech_Book2} 

\begin{definition}
Let $\kappa$ be inaccessible. Then
$$\Q_{<\kappa}=\setof{a\in V_k}{a \text{ is stationary and } \forall x \in a, x \text{ is countable.}}$$
\end{definition}

\begin{definition}
Let $\kappa$ be inaccessible and let $D\subset\Q_{<\kappa}$ be dense. Let $y$ be countable.
We say that $y$ \emph{captures} $D$ iff there is a $p\in D$ such that $p\in y$ and $y\intersect\union p \in p$.
\end{definition}

\begin{definition}
We say that $y$ \emph{end-extends} $x$ if $y\intersect V_{\rank(x)} = x$.
\end{definition}

\begin{definition}
Let $\kappa$ be inaccessible and let $D\subset\Q_{<\kappa}$ be dense. Then $\semiproper(D) = $ the set of countable
$x \prec V_{\kappa+1}$ such that there exists a countable
 $y \prec V_{\kappa+1}$ such that $x\subseteq y, y \text{ end-extends } x \intersect V_{\kappa}$ and  $y \text{ captures } D.$
\end{definition}

\begin{definition}
Let $\kappa$ be inaccessible and let $D\subset\Q_{<\kappa}$ be dense. We say that $D$ is \emph{semi-proper}
if $\semiproper(D)$ contains a club.
\end{definition}

\begin{lemma}
\label{SemiProperEquivalence}
Let $\kappa$ be inaccessible and let $D\subset\Q_{<\kappa}$ be dense and semi-proper. 
Let $\mu>\kappa$ be inaccessible.
Then for any countable $x\prec V_{\mu}$ such that $D\in x$, $\exists y \prec V_{\mu}$
s.t. $y$ is countable, $x\subseteq y, y \text{ end-extends } x \intersect V_{\kappa}$ and $ y \text{ captures } D$.
\end{lemma}
\begin{proof}
(cf Lemma 2.5.4 of \cite{Larson_Book} and Lemma 34.12 of \cite{Jech_Book2}.)
Let $x\prec V_{\mu}$ be countable with $D\in x$. Since $D$ is cofinal in $V_{\kappa}$, $\kappa\in x$.
Since $D$ is semi-proper,
there is an $f:V_{\kappa+1}^{<\omega}\map V_{\kappa+1}$ such that for all countable $z\subset V_{\kappa+1}$,
if $z$ is closed under $f$ then $z\in\semiproper(D)$. Since $x\prec V_{\mu}$, there is such an $f\in x$. Fix such an $f$.
Then $x\intersect V_{\kappa+1}$ is closed under $f$ so $x\intersect V_{\kappa+1} \in \semiproper(D)$. Let $\bar{y}$ witness this,
i.e. $\bar{y}$ is countable, $\bar{y}\prec V_{\kappa+1}, x\intersect V_{\kappa+1}\subseteq \bar{y}, \bar{y} \text{ end-extends } x \intersect V_{\kappa}$ and $\bar{y} \text{ captures } D$.
Let
$$y=\setof{h(w)}{\dom(h)=V_{\kappa} \AND h\in x \AND w\in \bar{y}\intersect V_{\kappa}}.$$
Then $y$ is countable,  $y \prec V_{\mu}$ and  $x\subset y$. Furthermore, since 
$x\intersect V_{\kappa+1} \subset \bar{y}$, $y\intersect V_{\kappa} = \bar{y} \intersect V_{\kappa}$.
So $y \text{ end-extends } x \intersect V_{\kappa}$ and $ y \text{ captures } D$.
\end{proof}

\begin{lemma}
\label{WoodinGivesSemiProper}
Let $\delta$ be Woodin. Let $D\subset \Q_{<\delta}$ be dense. Then there are cofinally many inaccessible $\kappa<\delta$ such
that $D\intersect V_{\kappa}$ is dense in $\Q_{<\kappa}$ and semi-proper.
\end{lemma}
\begin{proof}
This is Lemma 34.13 of \cite{Jech_Book2} and also follows from Theorem 2.7.5 of \cite{Larson_Book}.
\end{proof}

\begin{proof}[proof of Lemma  \ref{LadderGeneric}]
Let $M$ be a countable ladder-mouse
with $\sequence{\delta_n,\gamma_n}{n\in\omega}$ a ladder over $M$ and
$\gamma=\sup\setof{\gamma_n}{n\in\omega}$.
We must show that
there is a $G\subset \Q_{<\gamma}^M$ such that for $n\in\omega$,
$G\intersect J^M_{\delta_n}$ is $\Q_{<\delta_n}^M$-generic over $J^M_{\gamma_n}$.

\begin{claim}[Claim 1]
Fix $n\in\omega$ and $a\in \Q_{<\delta_n}^M$.
Let $\lambda = \delta_n + \omega$. Working in $M$,
let  $b=\{x \prec V_{\lambda} \mid x $ is countable and
$x\intersect\union a \in a$ and $\forall D\in x\intersect J^M_{\gamma_n}$ such that $D$ is a dense subset of $\Q_{<\delta_n}^M$, $x$ captures $D \}$.
Then $M\models b$ is stationary.
\end{claim}
\begin{subproof}[proof of Claim 1.]
Work in $M$. Let $f:V_{\lambda}^{<\omega}\map V_{\lambda}$, and we will
construct an $x\in b$ closed under $f$.

Let $\mu$ be the least inaccessible greater than $\delta_n$. Since $\delta_{n+1}$ is a limit
of inaccessibles in $M$, we have $\delta_n < \lambda < \mu < \delta_{n+1}$. We will find a countable
$x \prec V_{\mu}$ so that $f\in x$,
$x\intersect\union a \in a$ and $\forall D\in x\intersect J^M_{\gamma_n}$ such that $D$ is a dense subset of $\Q_{<\delta_n}^M$, $x$ captures $D$.
Then $x\intersect V_{\lambda} \in b$ and $x\intersect V_{\lambda}$ is closed under $f$.

We will construct $x$ as an increasing union $x = \Union{i} x_i$. 

Let $x_0 \prec V^M_{\mu}$ be countable, with $f\in x_0$ and such that $x_0\intersect \union a \in a$.

Let $\sequence{D_0^j}{j\in\omega}$ enumerate the dense subsets of $\Q_{<\delta_n}^M$ in $x_0\intersect J^M_{\gamma_n}$.

Temporarily working in $J^M_{\gamma_n}$ where $\delta_n$
is Woodin, Lemma \ref{WoodinGivesSemiProper} gives us that there are cofinally many inaccessible cardinals
$\kappa<\delta_n$ such
that $D^0_0\intersect V_{\kappa}$ is dense and semi-proper in $\Q_{<\kappa}$.
Fix such a $\kappa_0$ with $a\in V_{\kappa_0}$.

As the definition of semi-proper is local, it remains true in the full $M$ that
$\kappa_0$ is inaccessible and that $D^0_0\intersect V_{\kappa_0}$ is dense and semi-proper in $\Q_{<\kappa_0}$.
Now working in the full $M$ we apply Lemma \ref{SemiProperEquivalence} to conclude that there is a countable 
$x_1 \prec V_{\mu}$
s.t.  $x_0\subseteq x_1, x_1 \text{ end-extends } x_0 \intersect V_{\kappa_0}$ and $ x_1 \text{ captures } D_0^0 \intersect V_{\kappa_0}$.

Let $\sequence{D_1^j}{j\in\omega}$ enumerate the dense subsets of $\Q_{<\delta_n}^M$ in $x_1\intersect J^M_{\gamma_n}$.

Let $\pi(n) = (\pi_0(n),\pi_1(n))$ be a bijection from $\omega$ to $\omega \times \omega$ such that for all $n$, $\pi_0(n)\leq n$,
and $\pi(0)=(0,0)$.

We will continue to define $x_i, \kappa_i$ by induction.

Suppose that $t\geq 1$ and we have the following:
\begin{enumerate}
\item $x_0 \subseteq \cdots \subseteq x_t$, countable elementary submodels of $V_{\mu}$
\item $\sequence{D_i^j}{j\in\omega}$ enumerates the dense subsets of $\Q_{<\delta_n}^M$ in $x_i\intersect J^M_{\gamma_n}$, for $i=0,\cdots,t$.
\item $\kappa_i < \delta_n$ is inaccessible and $D_{\pi_0(i)}^{\pi_1(i)} \intersect V_{\kappa_i}$ is dense and semi-proper for $i < t$.
\item $x_{i+1}$ captures $D_{\pi_0(i)}^{\pi_1(i)}\intersect V_{\kappa_i}$ for $i < t$.
\item $x_{i+1}$ end-extends $x_i \intersect V_{\kappa_i}$ for $i<t$.
\item $\kappa_i > \kappa_{i-1}$ for $0<i<t$.
\end{enumerate}

We will define $\kappa_t$, $x_{t+1}$, and $\sequence{D_{t+1}^j}{j\in\omega}$.

Temporarily working in $J^M_{\gamma_n}$ where $\delta_n$
is Woodin, Lemma \ref{WoodinGivesSemiProper} gives us that there are cofinally many inaccessible cardinals
$\kappa<\delta_n$ such
that $D_{\pi_0(t)}^{\pi_1(t)}\intersect V_{\kappa}$ is dense and semi-proper in $\Q_{<\kappa}$.
Fix such a $\kappa_t$ with $\kappa_t > \kappa_{t-1}$.

As the definition of semi-proper is local, it remains true in the full $M$ that
$\kappa_t$ is inaccessible and that $D_{\pi_0(t)}^{\pi_1(t)}\intersect V_{\kappa_t}$ is dense and semi-proper in $\Q_{<\kappa_t}$.
Now working in the full $M$ we apply Lemma \ref{SemiProperEquivalence} to conclude that there is a countable 
$x_{t+1} \prec V_{\mu}$
s.t.  $x_t\subseteq x_{t+1}, x_{t+1} \text{ end-extends } x_t \intersect V_{\kappa_t}$ and $x_{t+1} \text{ captures } D_{\pi_0(t)}^{\pi_1(t)} \intersect V_{\kappa_t}$.

Let $\sequence{D_{t+1}^j}{j\in\omega}$ enumerate the dense subsets of $\Q_{<\delta_n}^M$ in $x_{t+1}\intersect J^M_{\gamma_n}$.

Let $x = \Union{i} x_i$. $x$ is countable, $x \prec V_{\mu}$ and $f\in x$.
$x$ end-extends $x_0\intersect V_{\kappa_0}$ so
$x\intersect\union a \in a$.
Let $D\in x\intersect J^M_{\gamma_n}$ be a dense subset of $\Q_{<\delta_n}^M$.
Then $D=D_i^j$ for some $i,j$. Let $t$ be such that $\pi(t) = (i,j)$.
Since $x_{t+1} \text{ captures } D \intersect V_{\kappa_t}$ and $\kappa_s > \kappa_t$ for $s>t$
and $x_{s+1}$ end-extends $x_s\intersect V_{\kappa_s}$ for $s>t$, we have that $x_s$ captures $D\intersect V_{\kappa_t}$ for all $s\geq t$, so $x$ captures $D$.
\end{subproof}

\begin{claim}[Claim 2]
For $n\in\omega$, for each $a\in \Q_{<\delta_n}^M$ there is a $b\in \Q_{<\delta_{n+1}}^M$ such
that $b < a$ and
$$b\forces\Gdot\intersect \Q_{<\deltacheck_n}\text{ is } J^M_{\gammacheck_n}\text{ generic.}$$
More precisely,
for all $D\in J^M_{\gamma_n}$ such that $D$ is a dense
subset of $\Q_{<\delta_n}^M$, and for all $c\in \Q_{<\gamma}^M$ with
$c\leq b$, there is a $d\subseteq c$, $d\in \Q_{<\gamma}^M$, and
there is an $e\in D$ such that $d < e$.
\end{claim}
\begin{subproof}[proof of Claim 2.]
Fix $n\in\omega$ and $a\in \Q_{<\delta_n}^M$. Let $b$ be as in Claim 1.
Then $b\in \Q_{<\delta_{n+1}}^M$ and
$b<a$. Fix $D\in J^M_{\gamma_n}$ such that $D$ is a dense
subset of $\Q_{<\delta_n}^M$, and fix $c\in \Q_{<\gamma}^M$ with
$c\leq b$.

Let $\dprime=\setof{x\in c}{D\in x}$. Then $\dprime$ is stationary
and for all $x\in\dprime$, since $x\intersect V^M_{\lambda}\in b$, (where $\lambda = \delta_n + \omega$)
$x$ captures $D$. There is an $e\in D$ such that, letting
$d=\setof{x\in\dprime}{x \text{ captures } D \text{ at } e}$, $d$ is stationary.
So $d<e$.
\end{subproof}

To finish the proof of Lemma  \ref{LadderGeneric}, let
$\sequence{D^n_k}{k\in\omega}$ enumerate the dense subsets of
$\Q_{<\delta_n}^M$ in $J^M_{\gamma_n}$, for each $n$. Define by induction on
$n$ a sequence of conditions $a_n\in\Q_{<\gamma}^M$ such that
\begin{itemize}
\item $a_n\in\Q^M_{<\delta_n}$,
\item $a_{n+1} < a_n$,
\item For $m,k < n$, there is an $e\in D^m_k$ such that $a_{n+1}<e$.
\end{itemize}
Claim 2 lets us construct such a sequence. (Notice that when defining
$a_{n+1}$ we have an obligation to ensure that $a_{n+1}<e$ for some
$e\in D^m_n$, for $m<n$, not only for $m=n$. But Claim 2 is strong enough
to allow this.)
Let $G=\setof{a\in \Q_{<\gamma}^M}{a\geq a_n \text{ for some } n}$.
Then for $n\in\omega$,
$G\intersect J^M_{\delta_n}$ is $\Q_{<\delta_n}^M$-generic over $J^M_{\gamma_n}$.

\end{proof}

\section{Beyond the first pointclass}
\label{section:beyond_first_pointclass}

By combining the ideas of this paper with those from
\cite{Mouse_Sets}, we believe that it should be possible to prove the mouse set theorem
for projective-like pointclasses beyond $\Pi^1_{\omega+1}$. In this section we
describe the relationship between this paper and \cite{Mouse_Sets} and we
sketch a proof of Theorem \ref{MRealsAreDefinable}.

Recall that $\Sigma^1_{\omega+n} = \Powerset(\R)\intersect \Sigma_{n+1}(J_2(\R))$.
In \cite{Mouse_Sets}, the pointclasses were named differently. In that paper
the following pointclasses were defined. For $\alpha\geq 2$,
$$\Sa{0}=\Powerset(\R)\intersect \Sigma_{1}(J_{\alpha}(\R)),$$
and for $n\geq 0$
\begin{align*}
\Pan &= \neg \San, \\
\Dan &= \San \intersect \Pan, \\
\Sa{n+1} &= \exists^{\R} \Pan.
\end{align*}

The $\Sigma^1_{\omega}$ of this paper is equal
to the $\Sigma_{(2,0)}$ of \cite{Mouse_Sets}, and the
$\Pi^1_{\omega+1}$ of this paper is
equal to the $\Pi_{(2,1)}$ of \cite{Mouse_Sets}.

The pointclasses $\San,\Pan,\Dan$ are only interesting for $\alpha$ such that
$\alpha$ begins a $\Sigma_1$-gap (see \cite{Scales_In_LofR})
and $\JalphaR$ is not admissible. In
\cite{Mouse_Sets} such $\alpha$ were called \emph{projective-like}.

For $\alpha$ projective-like and $n\geq 1$, \cite{Mouse_Sets} defines

$$\Aan\defeq\setof{x\in\R}{\exists \xi<\omega_1\, x\in \Dan(\xi)}$$

and defines

$$\Aa{0} = \Union{\beta<\alpha} \OD^{\JofR{\beta}}.$$

The $Q^1_{\omega+1}$ of this paper is equal to the $A_{(2,1)}$ of \cite{Mouse_Sets}.

In \cite{Mouse_Sets} also the mice were named differently. In that paper
a hierarchy of mouse-smallness properties called $(\alpha,n)$-\emph{petite}, for $\alpha\geq 1$
and $n\in\omega$ was defined. The $\Mladder$ of this paper is equal to the least mouse
$M$ that is not $(2,1)$-petite.

In \cite{Mouse_Sets} we show that if $M$ is iterable and $(\alpha,n)$-petite, then
$\R\intersect M \subseteq A_{(\alpha,n)}$. Taking $\alpha=2$ and $n=1$, and translating into the language of the current paper, we get that
$\R\intersect \Mladder \subseteq Q^1_{\omega+1}$, which is Theorem \ref{MRealsAreDefinable} of this paper.
Below we present a sketch of the proof of this theorem.
The interested reader can consult \cite{Mouse_Sets} for details.

\begin{definition}
A countable premouse $M$ is $n$-iterable iff
player II wins the weak iteration
game on $M$ of length $n$. (See \cite{Many_Woodins}.)
\end{definition}

\begin{definition}
An $\omega$-mouse is a sound premouse that projects to $\omega$.
\end{definition}

\begin{remarks}
\quad
\begin{itemize}
\item Let $P(x) \iff x\in\R $ codes a countable premouse $M$ and $M$ is $n$-iterable.
Then $P$ is projective.
\item By \cite{Proj_WO_In_Mod}, if $M$ and $N$ are $n$-small $\omega$-mice
and $M$ is fully iterable and $N$ is $n$-iterable, then $M\initseg N$ or $N\initseg M$.
\end{itemize}
\end{remarks}

\begin{definition}
\label{def-pi-one-omega-iterable}
A premouse $M$ is \emph{$\Pi^1_{\omega}$-iterable} iff $M$ is $n$-iterable for all $n$.
\end{definition}

\begin{remarks}
\quad
\begin{itemize}
\item  Let $P(x) \iff x\in\R $ codes a countable premouse $M$ and $M$ is $\Pi^1_{\omega}$-iterable.
Then $P$ is $\Pi^1_{\omega}$.
\item $\Pi^1_{\omega}$-iterable was called $\Pi_{(2,0)}$-iterable in \cite{Mouse_Sets}.
\end{itemize}
\end{remarks}

\begin{definition}
\label{def-ladder-small}
A premouse $M$ is \emph{ladder-small} iff no initial segment of $M$ is a ladder mouse.
\end{definition}

\begin{remark}
Ladder-small is the same as $(2,1)$-petite from \cite{Mouse_Sets}.
\end{remark}

\begin{lemma}
\label{comparison_lemma}
Suppose $M$ and $N$ are countable $\omega$-mice and $M$ is
ladder-small and fully iterable and $N$ is $\Pi^1_{\omega}$-iterable.
Then $M\initseg N$ or $N\initseg M$.
\end{lemma}
\begin{proof}[sketch of proof]
This is essentially part (3) of Theorem 7.16 of \cite{Mouse_Sets}, with $(\alpha,n)=(2,0)$
and $\delta=0$.
Here we give only a sketch.

Compare $M$ and $N$ forming iteration trees $\cT$ on $M$
and $\cU$ on $N$. On the $M$ side, at limits, choose branches by iterability.
On the $N$ side, at limits, choose the unique cofinal, wellfounded
branch $b$ such that $Q(b,\cU)$ is $\Pi^1_{\omega}$-iterable above $\delta(b,\cU)$,
if there is a unique such branch. If it ever occurs that there is not a
unique such branch, then we can show that for each $n$, there must be a maximal
branch $b_n$ such that $Q(b_n,\cU)$ is not $n$-small above $\delta(b_n,\cU)$.
Since $\delta(b_n, \cU)$ remains a cardinal, the process has generated a ladder.
This is a contradiction since we assumed that $M$ is ladder-small.
\end{proof}

\begin{proof}[Proof of Theorem \ref{MRealsAreDefinable}]
Let $x\in\R\intersect\Mladder$ and we will show that $x\in Q_{\omega+1}$. Let $\xi$
be the rank of $x$ in the order of construction of $\Mladder$. Let $N\initseg\Mladder$
be the $\initseg$-least initial segment of $M$ that contains $x$ and projects
to $\omega$. By Lemma \ref{comparison_lemma}, if $\Nprime$ is any $\omega$-mouse
and is $\Pi^1_{\omega}$-iterable, then $N\initseg\Nprime$
or $\Nprime\initseg N$. Thus $x$ is the unique real such that there is some
$\Pi^1_{\omega}$-iterable $\omega$-mouse $\Nprime$ such that
$x$ is the $\xi$th real in the order of construction of $\Nprime$. This gives
a definition for $\singleton{x}$ that is $\Sigma^1_{\omega+1}$ in any code
for $\xi$. So $x$ is $\Delta^1_{\omega+1}(\xi)$, so $x\in Q_{\omega+1}$.
\end{proof}

To formulate a mouse set theorem for larger pointclasses, one must give a definition
of the appropriate mouse for each pointclass. In \cite{Mouse_Sets} that was done
via the recursive definition of $(\alpha,n)$-petite.  \cite{Mouse_Sets} develops
the recursive machinery only up through the ordinal $\alpha=\omega_1^{\omega_1}$.
We believe, but we haven't checked carefully, that by combining the ideas of
this paper with those from
\cite{Mouse_Sets}, it should be possible to prove the following conjecture:

\begin{conjecture}
\label{PetiteIsRight}
Assume there are $\omega$ Woodin cardinals with a measurable cardinal above them all. Suppose that
$(2,0) \lexleq (\alpha,n) \lexleq (\omega_1^{\omega_1}, 0)$.
Let $M$ be the least fully iterable mouse that is not $(\alpha, n)$-petite.
Then $\R\intersect M = \Aan$.
\end{conjecture}

Corollary 7.17 of \cite{Mouse_Sets} gives $\R\intersect M \subseteq \Aan$
and Theorem 8.1 of \cite{Mouse_Sets} gives $\Aan \subseteq \R\intersect M$
in the case that $\cof(\alpha)>\omega$ or $n=0$. What was missing from the
induction in \cite{Mouse_Sets} was a proof that $\Aan\subseteq \R \intersect M$
in the case that $\cof(\alpha)=\omega$ (or $\alpha$ a successor ordinal) and $n>0$.
The ideas of the current paper
should be able to be generalized to fill that gap.

A warning for readers of \cite{Mouse_Sets}: That paper contains some material which, with the benefit of hindsight,
we now know is off-point. Because it was unknown whether or not
$\Aan\subseteq \R\intersect M$, where $M$ is the least mouse that is not
$(\alpha,n)$-petite, \cite{Mouse_Sets} defines a second hierarchy of mouse
smallness properties called $(\alpha,n)$-small. The least mouse that is not
$(\alpha,n)$-small is greater in the mouse hierarchy than the least mouse that
is not $(\alpha,n)$-petite. For example, the least mouse that is not
$(2,1)$-petite is $\Mladder$, but the least mouse that is not
$(2,1)$-small is $M^{\text{lda}}$, the minimal \emph{admissible} ladder mouse,
i.e. the minimal mouse $M$ such that $M$ is a ladder-mouse and
$M$ is admissible. Thus $\Mladder\properseg M^{\text{lda}}$.

In
\cite{Mouse_Sets} it was shown that
$\R\intersect \Mladder \subseteq Q_{\omega+1}\subseteq \R\intersect M^{\text{lda}}$, and in general that $\R \intersect M \subseteq \Aan\subseteq \R\intersect\Mprime$, where $M$ is the least mouse that is not
$(\alpha,n)$-petite and $\Mprime$ is the least mouse that is not
$(\alpha,n)$-small. It was then conjectured that
$\Aan=\R\intersect \Mprime$, where $\Mprime$ is the least mouse that is not
$(\alpha,n)$-small. But we now know that this conjecture was incorrect. The
current
paper shows that $Q_{\omega+1}=\R\intersect\Mladder$ and so the conjecture
that $Q_{\omega+1}=\R\intersect M^{\text{lda}}$ is false. In this paper we
replace the false conjecture with Conjecture \ref{PetiteIsRight}.

The reader of \cite{Mouse_Sets} should be warned that they will encounter
material about $(\alpha,n)$-\emph{big} premice
(i.e. premice which are not $(\alpha,n)$-small) which should essentially
be ignored as this material was introduced in support of the false conjecture.
In other words, because it was not known how to prove that $Q_{\omega+1}$
was contained in the minimal ladder mouse, \cite{Mouse_Sets} introduced the
notion of the minimal admissible ladder mouse and proved that
$Q_{\omega+1}$ was contained in that mouse. But the current paper gives a
proof that $Q_{\omega+1}$ is contained in the minimal ladder mouse and
that indicates
that the introduction of the minimal admissible ladder mouse was an
unnecessary distraction.

\section{Eliminating the ordinal parameter}
\label{section:eliminate_ordinal_parameters}

In this section we show that the ordinal parameter $\theta$ in Theorem
\ref{QuasiCorrectness} can be eliminated in the case that $M=\Mladder$.

The ideas and results in this section are due to John Steel.

In this section we will make use of some material from the previous section.
In particular to understand this section the reader should be familiar with
the Comparison Theorem for 
 $\Pi^1_{\omega}$-iterable ladder-small mice.
 See Definition \ref{def-pi-one-omega-iterable},
 Definition \ref{def-ladder-small} and Lemma \ref{comparison_lemma}.

To motivate this section, consider the following fact

\begin{fact}
Assume there is a Woodin cardinal with a measurable cardinal above it. To every $\Sigma^1_3$ formula $\varphi$ we can recursively
associate a formula $\sigma$ in the language of mice such that for all reals $x$ we have
\begin{align*}
\varphi(x) \Iff &\exists \text{ a } \Pi^1_2\text{-iterable, sound, countable } x\text{-mouse } M \text{ projecting to $\omega$ s.t. } \\
                & M\models\sigma[x].
\end{align*}               
\end{fact}

The statement ``$\exists \text{ a } \Pi^1_2\text{-iterable, sound, countable } x\text{-mouse } M$  projecting to $\omega$ s.t. 
$M\models\sigma[x]$'' is a $\Sigma^1_3$ statement.
So the above fact gives a canonical representation of all $\Sigma^1_3$ statements in terms of mice.

\begin{proof}[sketch of proof of the fact]
Let $\sigma$ be the formula such that for all countable, sound $x$-mice $M$ projecting to $\omega$, $M\models\sigma[x]$ iff
$M$ is not $1$-small but every initial segment of $M$ is $1$-small
and letting  $\delta$ be the unique Woodin cardinal of $M$
we have that $1$ forces that after collapsing
$\delta$ to be countable, $\varphi[x]$ holds.

The first clause of $\sigma(x)$ asserting that $M$ is not $1$-small but every initial segment of $M$ is $1$-small
can be thought of as the first-order statement ``I am $M_1^{\sharp}(x)$.''
Any countable, sound $x$-mouse $M$ that projects to $\omega$ such $M\models$``I am $M_1^{\sharp}(x)$'' is called a \emph{putative } $M_1^{\sharp}(x)$.
If $M$ is a putative $M_1^{\sharp}(x)$ then $M$ is equal to the real $M_1^{\sharp}(x)$ iff $M$ is fully iterable.

If $\varphi(x)$ holds in $V$ then $M_1^{\sharp}(x)\models \sigma[x]$. Conversely, if $M$ is a $\Pi^1_2$-iterable, sound $x$-mouse that projects to $\omega$ such
that $M\models\sigma[x]$, then we can simply iterate the last extender of $M$ $\omega_1$-times and apply Shoenfield absolutenss to conclude that $\varphi(x)$
is true in $V$.
\end{proof}

We would like to have have an analogous result for $\Mladder(x)$. That is, we would like to have the following

\begin{conjecture}
Assume there are $\omega$ Woodin cardinals with a measurable cardinal above them all. To every $\Sigma^1_{\omega+1}$ formula $\varphi$ we can recursively
associate a formula $\sigma$ in the language of mice such that for all reals $x$ we have
\begin{align*}
\varphi(x) \Iff &\exists \text{ a } \Pi^1_{\omega}\text{-iterable, sound, countable } x\text{-mouse } M \text{ projecting to $\omega$ s.t. } \\
                & M\models\sigma[x].
\end{align*}     
\end{conjecture}

The formula $\sigma$ from the conjecture should assert that $M$ is a ``putative $\Mladder(x)$'' and that $M\models\psi[x]$ where $\psi$ is the
formula given by theorem \ref{QuasiCorrectness}. The problem with this is that theorem \ref{QuasiCorrectness} includes a non-uniform ordinal
parameter which prevents us from giving a uniform canonical representation of $\Sigma^1_{\omega+1}$ forumulae in terms of mice.

In this section we will show how to eliminate that ordinal parameter. We will not give a proof of the conjecture here.

\begin{definition}
If $z\in\R$ then $\Mladder_z$ refers to the $z$-mouse which is like $\Mladder$ but built over $z$.
That is, $\Mladder_z$ is the $\initseg$-least fully-iterable $z$-mouse $M$ for which there is
a ladder over $M$.
\end{definition}

\begin{theorem}
\label{theta_is_definable}
Fix $z\in\R$, let $M = \Mladder_z$, and let  $\theta<\omega_2^M$ be the
minimum of the ranks of all
$\bPiOneOmega$ trees in $M$ that are  incorrectly wellfounded.
Then $\theta$ is definable in $J^M_{\omega_2^M}$, uniformly in $z$,
without parameters (other than the extender sequence of $M$.)
\end{theorem}

\begin{note}
See Remark \ref{pure-correctly-wellfounded} for
what it means to say that a tree is correctly or incorrectly wellfounded
without reference to a $\PiOneOmega$-code or real parameter.
\end{note}

The $\theta$ from Theorem \ref{theta_is_definable} is exactly the ordinal parameter
from the proof of theorem \ref{QuasiCorrectness}.

For simplicity we will ignore the real parameter $z$ and prove the theorem for $M=\Mladder$.

\begin{definition}
Let $\sequence{\sigma_n}{n\in\omega}$ enumerate the sentences in the language of premice.

For $n\in\omega$ and $w\in\WO$, let
$S_{n,w}=\big\{x\in \R \,\mid\, x$  codes a $\Pi^1_{\omega}$-iterable,
ladder-small $\omega$-mouse $M$ such that $\ord(M)=|w|$ and  $M\models\sigma_n\big\}$.

$S_{n,w}$ is $\Pi^1_{\omega}(w)$.

For $n\in\omega$ and $w\in\WO$, uniformly, fix $h_{n,w}$ a $\Pi^1_{\omega}(w)$-code for $S_{n,w}$
and let $T_{n,w}=T^{h_{n,w}}$.
\end{definition}

\begin{remark}
By Lemma \ref{comparison_lemma}, either $|w| < (\omega_1)^{\Mladder}$
in which case every element of $S_{n,w}$ codes the same countable level of $\Mladder$, or else
$|w| \geq (\omega_1)^{\Mladder}$ and every countable level of $\Mladder$ is an initial segment
of every $M$ coded by the elements of $S_{n,w}$.
\end{remark}

\begin{lemma}
\label{T_n_w_correctly_wellfounded}
Let $P\properseg J^{\Mladder}_{\omega_1^{\Mladder}}$ be an $\omega$-mouse. Let $w\in\WO$ with
$|w|=\ord(P)$. Fix $n\in\omega$
and suppose that $P\models \neg \sigma_n$. Then $T_{n,w}$ is wellfounded (in $V$.)
\end{lemma}
\begin{proof}
$p[T_{n,w}]=S_{n,w} = \emptyset$ because if $M$ is a $\Pi^1_{\omega}$-iterable $\omega$-mouse with
$\ord(M) = \ord(P)$, then by Lemma \ref{comparison_lemma}, $M=P$, so $M\not\models\sigma_n$.
\end{proof}

\begin{definition}
For any projectively correct mouse $M$ such that $M\models \omega_1$ exists,
let $\NSat^M=\big\{ (n,w) \, \mid \, n\in\omega$ and $w\in\WO$ with $|w|<\omega_1^{M}$ and
for $P=\cJ^M_{|w|}$, $P$ projects to $\omega$ and $P\models \neg \sigma_n \big\}$.

If $T=T^M_{n,w}$ for $(n,w)\in\NSat^M$ then we say that $T$ is an $\NSat$ tree of $M$.
\end{definition}

By Lemma \ref{T_n_w_correctly_wellfounded}, all of the $\NSat$ trees of $\Mladder$
are correctly wellfounded. (Again, see Remark \ref{pure-correctly-wellfounded} for
what it means to say that a tree is correctly or incorrectly wellfounded
without reference to a $\PiOneOmega$-code or real parameter.)
 The key to the proof of Theorem
 \ref{theta_is_definable} is that the ranks of the $\NSat$ trees in $\Mladder$ are cofinal
 in the supremum of the ranks of all correctly wellfounded trees in $\Mladder$.

\begin{lemma}
\label{Nsat_trees_are_cofinal}
Let $T\in\Mladder$ be a $\bPiOneOmega$-tree that is correctly wellfounded.
Then there is an $(n,w)\in \NSat^{\Mladder}$ such that $\rank(T) < \rank(T^{\Mladder}_{n,w})$.
\end{lemma}

First we give the proof of Theorem \ref{theta_is_definable} assuming Lemma
\ref{Nsat_trees_are_cofinal}.

\begin{proof}[proof of Theorem \ref{theta_is_definable}]
As mentioned above, for simplicity we will ignore the real parameter $z$ and prove
the theorem for $\Mladder$.

Let $\thetaprime = $ the supremum of the ranks of all $\bPiOneOmega$-trees
in $\Mladder$ that are correctly wellfounded. By Lemmas
\ref{T_n_w_correctly_wellfounded} and \ref{Nsat_trees_are_cofinal}, $\thetaprime =$
the supremum of the ranks of the $\NSat$ trees in $\Mladder$.
Since $\NSat^{\Mladder}$
is definable without parameters in $J^{\Mladder}_{\omega_2^{\Mladder}}$,
this gives a definition of $\thetaprime$ that is also
so definable. Then $\theta^{\Mladder}$ is the minimum rank of a wellfounded
$\bPiOneOmega$-tree $T\in\Mladder$ with $\rank(T)>\thetaprime$.
\end{proof}

\begin{definition}
\label{LadderStationaryTowerEmbedding}
Fix $z\in\R$ and let $G$ be given by Lemma \ref{LadderGeneric} for $\Mladder_z$. That is,
let $\gamma=\ord(\Mladder_z)$ and let $G\subset \Q_{<\gamma}^{\Mladder_z}$ be such that for
$n\in\omega$, $G\intersect J^{\Mladder_z}_{\delta_n}$ is
$\Q_{<\delta_n}^{\Mladder_z}$-generic over $J^{\Mladder_z}_{\gamma_n}$, where
$\sequence{\delta_n,\gamma_n}{n\in\omega}$ is a ladder over $\Mladder$.
Let $\pi:\Mladder_z\to\Mstar$ be the stationary tower ultrapower.
We will capture this setting by saying that  $\pi:\Mladder_z\to\Mstar=\Ult(\Mladder_z,G)$
is a \emph{ladder stationary tower embedding} (for $\sequence{\delta_n,\gamma_n}{n\in\omega}$.)

Let $g$ be $\Coll(\delta,\omega)$-generic over $\Mladder$ for some
$\delta\in\Mladder$.
$\Mladder[g]$ can be reorganized as $\Mladder_z$ for some $z\in\R$.
Let $\psi:\Mladder_z\to\Nstar=\Mladder_z[H]$ be a ladder stationary tower embedding.
We will capture this setting by saying that
$\psi:\Mladder[g]\to\Nstar=\Ult(\Mladder[g],H)$ is a ladder stationary tower embedding \emph{above $g$}.
\end{definition}

Let $\Mstar$ be as in the above definition and let $T\in\Mstar$ be a 
$\bPiOneOmega$-tree. In what follows it will be convenient for us to use 
the terminology
$T$ is \emph{correctly-wellfound} or \emph{incorrectly-wellfounded}. But
because $\Mstar$ is not a wellfounded, iterable mouse, 
Remark \ref{pure-correctly-wellfounded} does not apply to $\Mstar$ and
so we need to be
careful about what these terms mean. Notice that
Definition \ref{correctly-wellfounded} still applies to $\Mstar$.

\begin{definition}
Let $\Mstar$ be a not-necessarily wellfounded model that has a
rank initial segment $(V_{\lambda})^{\Mstar}\in\wfp(\Mstar)$ such that
$(V_{\lambda})^{\Mstar}\models\ZFC$ and  $(V_{\lambda})^{\Mstar}$ is projectively-correct.
Let $T\in(V_{\lambda})^{\Mstar}$ be such that $\Mstar\models ``T$ is a $\bPiOneOmega$-tree''. Then we will say that $T$ is \emph{correctly wellfounded} iff there
is a $\PiOneOmega$-code $h$ and a real $x\in(V_{\lambda})^{\Mstar}$
such that $\left((T^h(x))^{\Mstar},h,x,\Mstar\right)$ is correctly
wellfounded. This means that $T$ is wellfounded (in $V$ and so in $\Mstar$)
and $T^h(x)$ is wellfounded in $V$.

We will say that $T$ is \emph{incorrectly wellfounded} iff
there
is a $\PiOneOmega$-code $h$ and a real $x\in(V_{\lambda})^{\Mstar}$
such that $\left((T^h(x))^{\Mstar},h,x,\Mstar\right)$ is incorrectly
wellfounded. This means that $T$ is wellfounded (in $V$ and so in $\Mstar$)
and $T^h(x)$ is illfounded in $V$.
\end{definition}

\begin{remark}
Since Remark \ref{pure-correctly-wellfounded} does not necessarily apply
to $\Mstar$, we cannot rule out the possibility that $T$ is both
correctly and incorrectly wellfounded. By
Remark \ref{pure-correctly-wellfounded},this cannot happen if $\Mstar$
is an iterable ladder mouse.
\end{remark}

Now we turn to a discussion of the proof of Lemma \ref{Nsat_trees_are_cofinal}.
Let $T\in\Mladder$ be a $\bPiOneOmega$-tree that is correctly wellfounded.
Let $\pi:\Mladder\to\Mstar=\Ult(M,G)$ be a ladder stationary tower embedding.
Let $T^* = \pi(T)$. It suffices to show that
$\Mstar\models\text{``}\rank(T^*) < \rank(T_{n,w})$ for some $(n,w)\in \NSat^{\Mstar}$''.

Since $T$ is correctly wellfounded, $T^*$ is correctly wellfounded.
Notice that $\Mstar\models\text{``}T_{n,w}$ is wellfounded for all $(n,w)\in \NSat$''
since $\Mladder$ satisfies that sentence. To find an $(n,w)\in \NSat^{\Mstar}$ such that
$\Mstar\models\text{``}\rank(T^*) < \rank(T_{n,w})$'', we will look for an
$(n,w)\in \NSat^{\Mstar}$
such that $(T_{n,w})^{\Mstar}$ is incorrectly wellfounded or illfounded.
We can accomplish this because Lemma \ref{T_n_w_correctly_wellfounded} is not
true if we repalce $\Mladder$ with $\Mstar$. That is, $\Mstar$ has non-standard
countable initial segments.

Let $\Pstar$ be the least initial segment of $\Mstar$ such that
$\omega_1^{\Mladder}\in\Pstar$ and
$\Pstar$ projects to $\omega$. There is such a $\Pstar$ since
$\ord(\Mladder)=\omega_1^{\Mstar}$. Notice that $\Pstar$ is ladder-small and
$\Pi^1_{\omega}$-iterable since every countable initial segment of $\Mladder$
has that property and $\Mstar$ is $\Pi^1_{\omega}$-correct.

\begin{lemma}
\label{taller_mice_are_nonstandard}
Let $P$ be a  $\Pi^1_{\omega}$-iterable,
ladder-small $\omega$-mouse with $\omega_1^{\Mladder}\in P$.
Let $w\in\WO$ be such that $|w|=\ord(P)$.
Then there is an $n$ such that $P\models \neg\sigma_n$ and $T_{n,w}$ is illfounded.
\end{lemma}
\begin{proof}
If the statement of the lemma were not true, then we would have that for all $n$,
the following are equivalent
\begin{enumerate}
\item[(a)] $P\models\sigma_n$
\item[(b)] There is a $\Pi^1_{\omega}$-iterable, ladder small $\omega$-mouse $N$,
with $\ord(N)=|w|$, such that $N\models\sigma_n$.
\item[(b)] For all $\Pi^1_{\omega}$-iterable, ladder small $\omega$-mice $N$,
with $\ord(N)=|w|$, $N\models\sigma_n$.
\end{enumerate}
This means that the theory of $P$ is $\DeltaOneOmegaPlusOne(w)$. By Theorem
\ref{DefinableRealsAreInM}, the theory of $P$ is in $\Mladder$. But $P$ is
sound and projects to $\omega$ so it is coded by its theory. So
$P\in\Mladder$. This is a contradiction since $P$ would be countable in
$\Mladder$ but $\omega_1^{\Mladder}\in P$.
\end{proof}

Let $w\in\WO\intersect\Mstar$ with $|w|=\ord(\Pstar)$ and let
$n$ be given by Lemma \ref{taller_mice_are_nonstandard}. Then
$(n,w)\in \NSat^{\Mstar}$ and
$(T_{n,w})^{\Mstar}$ is incorrectly wellfounded or illfounded. If we could apply
Theorem \ref{CorrectBelowIncorrect} to $\Mstar$ we could now conclude
that $\rank(T^*) < \rank(T_{n,w})^{\Mstar}$. But the proof of
Theorem \ref{CorrectBelowIncorrect} does not apply to $\Mstar$ because it is
not an iterable ladder-mouse.

Suppose $T,S$ are $\bPiOneOmega$-trees in $\Mstar$ and $T$ is correctly wellfounded
and $S$ is incorrectly wellfounded. If $T,S\in\ran(\pi)$, say
$T=\pi(\bar{T}),S=\pi(\bar{S})$, then $\bar{T}$ is correctly wellfounded and
$\bar{S}$ is incorrectly wellfounded so $\rank(\bar{T})<\rank(\bar{S})$ so
then we can conclude that $\rank(T)<\rank(S)$.

Now $T^*\in\range(\pi)$ but $T_{n,w}\not\in\range(\pi)$ because $w\not\in\Mladder$. However
$w$ can be chosen to be generic over $\Mladder$ and it turns out that is enough to see that
$\rank(T^*)<\rank(T_{n,w})$.

\begin{lemma}
\label{StarCorrectBelowIncorrect}
Let $\pi:\Mladder\to\Mstar=\Ult(\Mladder,G)$ be a ladder stationary tower embedding.
Let $T,S$ be $\bPiOneOmega$-trees in $\Mstar$.
Suppose that $T$ is correctly wellfounded and
$S$ is incorrectly wellfounded. Suppose further that $T$ is $\PiOneOmega(x)$
and $S$ is $\PiOneOmega(y)$ for $x,y\in\R\intersect\Mladder[g]$,
where $g$ is $\Coll(\delta,\omega)$ generic over $\Mladder$ for some $\delta\in\Mladder$, and there are $\PiOneOmega$ codes $h_1, h_2$ such that
$(T,h_1,x,\Mstar)$ is correctly wellfounded and 
$(S,h_2,y,\Mstar)$ is incorrectly wellfounded. Then $\rank(T) < \rank(S)$.
\end{lemma}

Towards a proof of Lemma \ref{StarCorrectBelowIncorrect},
fix $\PiOneOmega$ codes $h_1, h_2$
such that $T=(T^{h_1}(x))^{\Mstar}$ and
$S=(T^{h_2}(y))^{\Mstar}$ and $(T,h_1,x,\Mstar)$ is correctly wellfounded and 
$(S,h_2,y,\Mstar)$ is incorrectly wellfounded.
Then $(T^{h_1}(x))^{\Mladder[g]}$
is correctly wellfounded and
$(T^{h_2}(y))^{\Mladder[g]}$
is incorrectly wellfounded, or illfounded.
Because $\Mladder[g]$ can be rearranged as $\Mladder_z$ for some $z\in\R$,
we can apply Theorem \ref{CorrectBelowIncorrect} to conclude that
$\Mladder[g]\models \text{``}\rank((T^{h_1}(x))) < \rank((T^{h_2}(y)))$".
Let $\psi:\Mladder[g]\map\Nstar=\Ult(\Mladder[g],H)$ be a ladder stationary tower embeddding
above $g$. Then $\Nstar\models \text{``}\rank((T^{h_1}(x))) < \rank((T^{h_2}(y)))$".
We want to see that $\Mstar\models \text{``}\rank((T^{h_1}(x))) < \rank((T^{h_2}(y)))$".
This turns out to be true because $(T^{h_1}(x))^{\Nstar}=(T^{h_1}(x))^{\Mstar}$
and $(T^{h_2}(y))^{\Nstar}=(T^{h_2}(y))^{\Mstar}$. To see that this is true, we
now turn to an alternative characterization of these trees that is independent
of stationary tower extensions and depends only on symmetric collapses.

\begin{lemma}
\label{applyhjorthcorollary}
Let $M$ be a countable, transitive model of $\ZFC$ and  $\delta>\omega$ a cardinal of $M$.
Suppose that whenever $\P\in V^M_{\delta+1}$ is a poset and $g$ is $\P$-generic over $M$, then
$M[g]$ is projectively-correct. Let $g_1$ and $g_2$ be $\Coll(\delta,\omega)$-generic over $M$.
Let $x\in\R\intersect M$.

Let $\varphi$ be a $\Sigma^1_n(x)$ norm for some $n$.
Let $y_1\in\R\intersect M[g_1]$. Then there is a $y_2\in\R\intersect M[g_2]$ such
that $\varphi(y_1)=\varphi(y_2)$, and $\varphi^{M[g_1]}(y_1)=\varphi^{M[g_2]}(y_2)$.

Furthermore $\sigma^{M[g_1]} =\sigma^{M[g_2]}$ where
$$\sigma^{M[g_i]}:(\bdelta^1_{\omega})^{M[g_i]}\map \bdelta^1_{\omega}$$
is given by Definition \ref{norm_embedding_def}.

Let $h$ be a $\PiOneOmega$ code for a subset
of $\R^2$. Then $(T^h(x))^{M[g_1]}=(T^h(x))^{M[g_2]}$
and $\sigma^{M[g_1],h,x} = \sigma^{M[g_2],h,x}$ where
$$\sigma^{M[g_i],h,x}:(T^h(x))^{M[g_i]}\map(T^h(x))^V$$
is the tree embedding from
Definition \ref{norm_embedding_def}.
\end{lemma}
\begin{proof}
All of this follows easily from Corollary \ref{hjorthcorollary} and its proof.

Let $y_1\in\R\intersect M[g_1]$ be as in the statement of the Lemma. The proof of
Corollary \ref{hjorthcorollary} gives us a
$y_2\in\R\intersect M[g_2]$ such that $\varphi(y_1)=\varphi(y_2)$.

Since this is true for all projective norms $\varphi^{\prime}$
we have that
$\ran(\sigma^{M[g_1]})= \ran(\sigma^{M[g_2]})$
and since $\sigma^{M[g_i]}$ is the inverse of the transitive collapse of its range, we have that
$\sigma^{M[g_1]}=\sigma^{M[g_2]}$ and so $\varphi^{M[g_1]}(y_1)=\varphi^{M[g_2]}(y_2)$.

Suppose $(s,u)\in (T^h(x))^{M[g_1]}$. Let $n=\length(s)=\length(u)$. Then there is
$y_1\in\R\intersect M[g_1]$ extending $s$ such that $(\forall i<n)\, (y_1,x)\in G^{i}_{h(i)}$ and
$u(i)=(\varphi^{i_0}_{h(i_0),i_1}(y_1,x))^{M[g_1]}$.

By Corollary \ref{hjorthcorollary}, relativized to $x$, there is a $y_2\in\R\intersect M[g_2]$ such that
\begin{itemize}
\item $y_2 \restr n = y_1 \restr n$,
\item $(y_2,x)\in \Intersection{i<n} G^i_{h(i)}$, and
\item for $i<n$, $\varphi^{i_0}_{h(i_0),i_1}(y_2,x) = \varphi^{i_0}_{h(i_0),i_1}(y_1,x)$.
\end{itemize}

So then for $i<n$, $(\varphi^{i_0}_{h(i_0),i_1}(y_2,x))^{M[g_2]} =
(\varphi^{i_0}_{h(i_0),i_1}(y_1,x))^{M[g_1]}$.
So $(s,u)\in (T^h(x))^{M[g_2]}$. We have shown that $(T^h(x))^{M[g_1]}\subseteq(T^h(x))^{M[g_2]}$
and the same argument shows that
$(T^h(x))^{M[g_2]}\subseteq(T^h(x))^{M[g_1]}$.

That $\sigma^{M[g_1],h,x} = \sigma^{M[g_2],h,x}$ now follows immediately.
\end{proof}

\begin{definition}
\label{TInfinity}
Let $M$ be a countable transitive model of
$$\ZFC - \text{ Replacement } +
\text{ ``} \exists \text{ cofinally many inaccessible cardinals.'' }$$
Suppose that every set-generic extension of $M$ is projectively correct.

Let $h$ be a $\PiOneOmega$ code for a subset
of $\R^2$ and let $x\in\R\intersect M$.

Let $\delta$ be a cardinal of $M$. Then we define
\begin{itemize}
\item $X^{M,\delta} = \ran(\sigma^{M[g]})$
\item $\bar{T}^{h,x,M,\delta}=\ran(\sigma^{M[g],h,x})$
\end{itemize}
where $g$ is any
$\Coll(\delta,\omega)$-generic over $M$.

The definitions makes sense because by Lemma \ref{applyhjorthcorollary},
applied to a rank initial segment of $M$,
$\sigma^{M[g]}$ and $\sigma^{M[g],h,x}$ don't depend on the choice of $g$.

$X^{M,\delta}\subset \bdelta^1_{\omega}$, namely
$X^{M,\delta}=\big\{ \varphi(y)\, \mid \, \varphi$ is a projective norm and $y\in\R\intersect M[g]$ for
some $g\subset\Coll(\delta,\omega)$ generic over $M \big\}$.

$\bar{T}^{h,x,M,\delta}\subset (T^h(x))^{V}$, namely
$\bar{T}^{h,x,M,\delta}=\big\{(s,u)\in (T^h(x))^{V} \, \mid \,
\exists g\, \Coll(\delta,\omega)$-generic over $M$ and $\exists y \in \R\intersect M[g]$ s.t.
$y$ extends $s$ and $(\forall i<\length(u))\, (y,x)\in G^{i}_{h(i)}$ and
$u(i)=\varphi^{i_0}_{h(i_0),i_1}(y,x) \big\}$.

If $\delta_1<\delta_2$ then $X^{M,\delta_1}\subset X^{M,\delta_2}$ and
$\bar{T}^{h,x,M,\delta_1}\subset \bar{T}^{h,x,M,\delta_2}$.

Let

$$X^{M,\infty}=\bigcup\setof{X^{M,\delta}}{\delta\in M}$$
$$\bar{T}^{h,x,M,\infty}=\bigcup\setof{\bar{T}^{h,x,M,\delta}}{\delta\in M}.$$

Let $\sigma^{M,\infty}$ be the inverse of the transitive collapse of $X^{M,\infty}$.

Let
$$T^{h,x,M,\infty}=\setof{(s,u)}{(s,\sigma^{M,\infty}\circ u)\in \bar{T}^{h,x,M,\infty}}.$$
\end{definition}

Very roughly speaking $T^{h,x,M,\infty}$ is ``the version of $T^h(x)$ in the symmetric collapse
over $M$ of the ordinals of $M$.''

One example of a model $M$ that satisfies the hypotheses of Definition \ref{TInfinity} is
$\Mladder[g]$ for some generic $g$. The next lemma states that in this case
a $\PiOneOmega(x)$-tree in the stationary tower ultrapower is equal to
$T^{h,x,\Mladder[g],\infty}$.

\begin{lemma}
\label{TStarIsTInfinity}
Let $\pi:\Mladder\to\Mstar=\Ult(\Mladder,G)$ be a ladder stationary tower embedding.
Let $g\in\Mstar$ be $\Coll(\delta,\omega)$-generic over $\Mladder$ for some cardinal $\delta\in\Mladder$.
Let $h$ be a $\PiOneOmega$ code for a subset
of $\R^2$ and let $x\in\R\intersect \Mladder[g]$.
Then $(T^h(x))^{\Mstar} = T^{h,x,\Mladder[g],\infty}$.
\end{lemma}
\begin{proof}
Let $\sequence{\delta_n,\gamma_n}{n\in\omega}$ be  ladder over $\Mladder$ such that
$\pi:\Mladder\to\Mstar=\Ult(\Mladder,G)$ is a ladder stationary tower embedding for $\sequence{\delta_n,\gamma_n}{n\in\omega}$.

Fix $m\in\omega$ and let  $y\in\R\intersect\Mstar$ be such that
$(y,x)\in \Intersection{i<m} G^i_{h(i)}$.

Let $n$ be even and large enough that $2m+1 < n$, $\delta < \delta_n$ and $g,y,x\in \Ult(J^{\Mladder}_{\gamma_n},G\intersect J^{\Mladder}_{\delta_n})$.
There is such an $n$ by the proof of Claim 2 in the proof of Lemma \ref{m_star_existence} where we showed that $\Mstar$ is projectively correct.

There is a $k$, $\Coll(\delta_n,\omega)$-generic over $J^{\Mladder}_{\gamma_n}$ such that $g,y,x\in J^{\Mladder}_{\gamma_n}[k]$.
By general forcing theory there is a $k_1$ such that $J^{\Mladder}_{\gamma_n}[k]=J^{\Mladder}_{\gamma_n}[g][k_1]$ with $k_1$
$\Coll(\delta_n,\omega)$-generic over $J^{\Mladder}_{\gamma_n}[g]$.

Let $k_2$ be
$\Coll(\delta_n,\omega)$-generic over $\Mladder[g]$. In particular $k_2$ is
$\Coll(\delta_n,\omega)$-generic over $J^{\Mladder}_{\gamma_n}[g]$.

By Corollary \ref{hjorthcorollary}, applied to $J^{\Mladder}_{\gamma_n}[g]$, and relativized to $x$,
there is a $y_2\in J^{\Mladder}_{\gamma_n}[g][k_2]$ such that
\begin{itemize}
\item $y_2 \restr m = y \restr m$,
\item $(y_2,x)\in \Intersection{i<m} G^i_{h(i)}$, and
\item for $i<m$, $\varphi^{i_0}_{h(i_0),i_1}(y_2,x) = \varphi^{i_0}_{h(i_0),i_1}(y,x)$.
\end{itemize}

Since $y_2\in \Mladder[g][k_2]$,
we have shown that $\ran(\sigma^{\Mstar,h,x})\subseteq \bar{T}^{h,x,\Mladder[g],\infty}$.

Conversely, suppose that $k$ is $\Coll(\delta,\omega)$-generic over $\Mladder[g]$ for some cardinal $\delta$ of $\Mladder[g]$ and
$y\in\R\intersect \Mladder[g][k]$ with $(y,x)\in \Intersection{i<m} G^i_{h(i)}$. Let $k_2\in\Mstar$ be
$\Coll(\delta,\omega)$-generic over $\Mladder[g]$. Then there is a $y_2\in \Mladder[g][k_2]$ such that
\begin{itemize}
\item $y_2 \restr m = y \restr m$,
\item $(y_2,x)\in \Intersection{i<m} G^i_{h(i)}$, and
\item for $i<m$, $\varphi^{i_0}_{h(i_0),i_1}(y_2,x) = \varphi^{i_0}_{h(i_0),i_1}(y,x)$.
\end{itemize}

Since $y_2\in \Mstar$,
we have shown that $\ran(\sigma^{\Mstar,h,x}) = \bar{T}^{h,x,\Mladder[g],\infty}$.

Since the previous argument may be applied to any projective norm $\varphi$ we also have that
$\ran(\sigma^{\Mstar}) =X^{\Mladder[g],\infty}$.

So then we get that $(T^h(x))^{\Mstar} = T^{h,x,\Mladder[g],\infty}$.

\end{proof}

\begin{corollary}
\label{TInMstarIsTInNstar}
Let $\pi:\Mladder\to\Mstar=\Ult(\Mladder,G)$ be a ladder stationary tower embedding.
Let $g\in\Mstar$ be $\Coll(\delta,\omega)$-generic over $\Mladder$ for some cardinal $\delta\in\Mladder$.
Let $h$ be a $\PiOneOmega$ code for a subset
of $\R^2$ and let $x\in\R\intersect \Mladder[g]$.

Let $\psi:\Mladder[g]\to\Nstar=\Ult(\Mladder[g],H)$ be a ladder stationary tower embedding above $g$.

Then $(T^h(x))^{\Mstar} = (T^h(x))^{\Nstar}$.
\end{corollary}
\begin{proof}
We can apply Lemma \ref{TStarIsTInfinity} twice to conclude that $(T^h(x))^{\Mstar} = T^{h,x,\Mladder[g],\infty} = (T^h(x))^{\Nstar}$.
\end{proof}

As we suggested above, we can use Corollary \ref{TInMstarIsTInNstar} to show that the rank of a correctly wellfounded tree is less
than the rank of an incorrectly wellfounded tree in $\Mstar$.

\begin{proof}[Proof of Lemma \ref{StarCorrectBelowIncorrect}]
Let $\pi:\Mladder\to\Mstar=\Ult(\Mladder,G)$ be a ladder stationary tower embedding.
Let $g\in\Mstar$ be $\Coll(\delta,\omega)$-generic over $\Mladder$ for some cardinal $\delta\in\Mladder$.
Let $h_1$ and $h_2$ be $\PiOneOmega$ codes for a subsets
of $\R^2$ and let $x,y\in\R\intersect \Mladder[g]$. Suppose that
$(T^{h_1})^{\Mstar}(x)$ is correctly wellfounded and $(T^{h_2})^{\Mstar}(y)$ is incorrectly wellfounded.
We must show that $\rank((T^{h_1})^{\Mstar}(x))<\rank((T^{h_2})^{\Mstar}(y))$.

$(T^{h_1}(x))^{\Mladder[g]}$
is correctly wellfounded and
$(T^{h_2}(y))^{\Mladder[g]}$
is incorrectly wellfounded, or illfounded. (It will follow from our proof that it is not illfounded.)
Because $\Mladder[g]$ can be rearranged as $\Mladder_z$ for some $z\in\R$,
we can apply Theorem \ref{CorrectBelowIncorrect} to conclude that
$\Mladder[g]\models \text{``}\rank(T^{h_1}(x)) < \rank(T^{h_2}(y))$". (For convenience we take the previous
sentence to include the possibility that $\rank(T^{h_2}(y)=\infty$.)
Let $\psi:\Mladder[g]\map\Nstar=\Ult(\Mladder[g],H)$ be a ladder stationary tower embeddding
above $g$. Then $\Nstar\models \text{``}\rank(T^{h_1}(x)) < \rank(T^{h_2}(y))$" and
so $\rank((T^{h_1})^{\Nstar}(x)) < \rank((T^{h_2})^{\Nstar}(y))$.
By Corollary \ref{TInMstarIsTInNstar},$(T^{h_1})^{\Nstar}(x)=(T^{h_1})^{\Mstar}(x)$ and
$(T^{h_2})^{\Nstar}(y)=(T^{h_2})^{\Mstar}(y)$.
So $\rank((T^{h_1})^{\Mstar}(x))<\rank((T^{h_2})^{\Mstar}(y))$.
\end{proof}

What remains in this section is to give the
\begin{proof}[Proof of Lemma \ref{Nsat_trees_are_cofinal}]
The proof was essentially given already in the discussion following Definition
\ref{LadderStationaryTowerEmbedding}.

Let $T\in\Mladder$ be a $\bPiOneOmega$-tree that is correctly wellfounded.
We need to find an $(n,w)\in \NSat^{\Mladder}$ such that $\rank(T) < \rank(T^{\Mladder}_{n,w})$.

Let $\pi:\Mladder\to\Mstar=\Ult(M,G)$ be a ladder stationary tower embedding.
Let $T^* = \pi(T)$. It suffices to show that
$\Mstar\models\text{``}\rank(T^*) < \rank(T_{n,w})$ for some $(n,w)\in \NSat^{\Mstar}$''.

By Lemma \ref{StarCorrectBelowIncorrect} it suffices to find
$(n,w)\in \NSat^{\Mladder}$ such that $(T_{n,w})^{\Mstar}$ is incorrectly wellfounded.

Let $\Pstar$ be the least initial segment of $\Mstar$ such that
$\omega_1^{\Mladder}\in\Pstar$ and
and $\Pstar$ projects to $\omega$.
Let $w\in\WO\intersect\Mstar$ with $|w|=\ord(\Pstar)$ and with $w\in \Mladder[g]$ for
some $g$ that is $\Coll(\delta,\omega)$-generic over $\Mladder$ for some cardinal $\delta$ of $\Mladder$.
Let
$n$ be given by Lemma \ref{taller_mice_are_nonstandard}. Then
$(n,w)\in \NSat^{\Mstar}$ and
$(T_{n,w})^{\Mstar}$ is incorrectly wellfounded.
\end{proof}

\bibliographystyle{amsalpha}
\bibliography{math}

\end{document}

%% file: mathdefs.tex
\newcommand{\setof}[2]{\left\{ \, #1 \, \left| \, #2 \, \right.\right\}}

\newcommand{\sequence}[2]{\left\langle \, #1 \,\left| \, #2 \, \right. \right\rangle}

\newcommand{\singleton}[1]{\left\{#1\right\}}
\newcommand{\angles}[1]{\left\langle #1 \right\rangle}

\newcommand{\force}[1]{\Vert\!\underset{\!\!\!\!\!#1}{\!\!\!\relbar\!\!\!%
\relbar\!\!\relbar\!\!\relbar\!\!\!\relbar\!\!\relbar\!\!\relbar\!\!\!%
\relbar\!\!\relbar\!\!\relbar}}

\newcommand{\forcein}[2]{\overset{#2}{\Vert\underset{\!\!\!\!\!#1}%
{\!\!\!\relbar\!\!\!\relbar\!\!\relbar\!\!\relbar\!\!\!\relbar\!\!\relbar\!%
\!\relbar\!\!\!\relbar\!\!\relbar\!\!\relbar\!\!\relbar\!\!\!\relbar\!\!%
\relbar\!\!\relbar}}}

\newcommand{\pre}[2]{{}^{#2}{#1}}

\newcommand{\restr}{\!\!\upharpoonright\!}

\newcommand{\intersect}{\cap}
\newcommand{\union}{\cup}
\newcommand{\Intersection}[1]{\bigcap\limits_{#1}}
\newcommand{\Union}[1]{\bigcup\limits_{#1}}

\newcommand{\forces}{\Vdash}

\newcommand{\defeq}{=_{\text{def}}}

\newcommand{\lexleq}{\leq_{\text{lex}}}
\newcommand{\lexless}{<_{\text{lex}}}

\newcommand{\R}{\mathbb{R}}
\renewcommand{\P}{\mathbb{P}}
\newcommand{\Q}{\mathbb{Q}}

\newcommand{\Bairespace}{\pre{\omega}{\omega}}
\newcommand{\LofR}{L(\R)}
\newcommand{\JofR}[1]{J_{#1}(\R)}

\newcommand{\JalphaR}{\JofR{\alpha}}

\DeclareMathOperator{\WO}{WO}
\DeclareMathOperator{\OD}{OD}

\DeclareMathOperator{\wfp}{wfp}

\newcommand{\Mladder}{M^{\text{ld}}}

\DeclareMathOperator{\ZFC}{ZFC}

\DeclareMathOperator{\AD}{AD}

\DeclareMathOperator{\PD}{PD}

\DeclareMathOperator{\GCH}{GCH}
\DeclareMathOperator{\dom}{dom}
\DeclareMathOperator{\ran}{ran}
\DeclareMathOperator{\range}{ran}

\DeclareMathOperator{\crit}{crit}

\DeclareMathOperator{\cof}{cof}
\DeclareMathOperator{\rank}{rank}
\DeclareMathOperator{\ot}{o.t.}

\DeclareMathOperator{\Powerset}{\mathcal{P}}
\DeclareMathOperator{\length}{lh}
\DeclareMathOperator{\lh}{lh}

\DeclareMathOperator{\Ult}{Ult}

\DeclareMathOperator{\Coll}{Coll}

\DeclareMathOperator{\semiproper}{sp}
\newcommand{\Implies}{\Rightarrow}
\newcommand{\Iff}{\Leftrightarrow}

\renewcommand{\iff}{\leftrightarrow}
\newcommand{\AND}{\wedge}
\newcommand{\OR}{\vee}

\newcommand{\map}{\rightarrow}

\newcommand{\initseg}{\trianglelefteq}
\newcommand{\properseg}{\lhd}

\newcommand{\cJ}{\mathcal{J}}

\newcommand{\cT}{\mathcal{T}}
\newcommand{\cU}{\mathcal{U}}

\newcommand{\dprime}{d^{\prime}}

\newcommand{\xprime}{x^{\prime}}
\newcommand{\yprime}{y^{\prime}}

\newcommand{\Mprime}{M^{\prime}}
\newcommand{\Nprime}{N^{\prime}}

\newcommand{\gammaprime}{\gamma^{\prime}}

\newcommand{\deltaprime}{\delta^{\prime}}

\newcommand{\thetaprime}{\theta^{\prime}}

\newcommand{\ibar}{\bar{i}}
\newcommand{\jbar}{\bar{j}}

\newcommand{\pibar}{\bar{\pi}}

\newcommand{\Mstar}{M^{*}}
\newcommand{\Nstar}{N^{*}}
\newcommand{\Pstar}{P^{*}}

\newcommand{\Gdot}{\dot{G}}

\newcommand{\deltacheck}{\check{\delta}}
\newcommand{\gammacheck}{\check{\gamma}}

\newcommand{\bPi}{\boldsymbol{\Pi}}

\newcommand{\bdelta}{\boldsymbol{\delta}}

\newcommand{\Sa}[2][\alpha]{\Sigma_{(#1,#2)}}
\newcommand{\Pa}[2][\alpha]{\Pi_{(#1,#2)}}
\newcommand{\Da}[2][\alpha]{\Delta_{(#1,#2)}}
\newcommand{\Aa}[2][\alpha]{A_{(#1,#2)}}

\newcommand{\San}{\Sa{n}}
\newcommand{\Pan}{\Pa{n}}
\newcommand{\Dan}{\Da{n}}

\newcommand{\Aan}{\Aa{n}}

\newcommand{\SigmaOneOmega}{\Sigma^1_{\omega}}
\newcommand{\SigmaOneOmegaPlusOne}{\Sigma^1_{\omega+1}}
\newcommand{\PiOneOmega}{\Pi^1_{\omega}}
\newcommand{\bPiOneOmega}{\bPi^1_{\omega}}
\newcommand{\PiOneOmegaPlusOne}{\Pi^1_{\omega+1}}
\newcommand{\DeltaOneOmegaPlusOne}{\Delta^1_{\omega+1}}

\DeclareMathOperator{\ord}{ord}

%% file: theoremstyles.tex
\newtheorem{theorem}{Theorem}[section]
\newtheorem{lemma}[theorem]{Lemma}
\newtheorem{corollary}[theorem]{Corollary}
\newtheorem{proposition}[theorem]{Proposition}

\theoremstyle{definition}

\newtheorem{definition}[theorem]{Definition}
\newtheorem{conjecture}[theorem]{Conjecture}
\newtheorem{remark}[theorem]{Remark}
\newtheorem{remarks}[theorem]{Remarks}

\theoremstyle{remark}

\newtheorem*{note}{Note}

\newtheorem*{fact}{Fact}

\newenvironment*{subproof}[1][Proof]
{\begin{proof}[#1]}{ \end{proof}}

\newenvironment*{claim}[1][Claim]
{\textbf{#1.  }\itshape }{}